\documentclass[a4paper]{article}

\usepackage{
amsmath,
amsthm,
amscd,
amssymb,
}
\usepackage{comment}
\usepackage{tikz}
\usetikzlibrary{positioning,arrows,calc}
\usepackage{xspace}

\usepackage{graphicx}
\usepackage[cal=boondoxo]{mathalfa}

\usepackage{url}

\newtheorem{maintheorem}{Theorem}

\theoremstyle{plain}
\newtheorem{theorem}{Theorem}[section]
\newtheorem{lemma}[theorem]{Lemma}
\newtheorem{definition}[theorem]{Definition}
\newtheorem{corollary}[theorem]{Corollary}
\newtheorem{proposition}[theorem]{Proposition}
\newtheorem{conjecture}[theorem]{Conjecture}
\newtheorem{question}[theorem]{Question}
\newtheorem{remark}[theorem]{Remark}

\newtheoremstyle{derp}
{3pt}
{3pt}
{}
{}
{\upshape}
{:}
{.5em}
{}
\theoremstyle{derp}

\newcommand{\R}{\mathbb{R}}

\newcommand{\Z}{\mathbb{Z}}

\newcommand{\N}{\mathbb{N}}

\newcommand{\poly}{\mathrm{poly}}
\newcommand{\counters}{\mathrm{idx}}

\newcommand\car{\mathbin{\char`\^}}

\newcommand{\Aut}{\mathrm{Aut}}
\newcommand{\PAut}{\mathrm{PAut}}

\newcommand{\Alt}{\mathrm{Alt}}
\newcommand{\id}{\mathrm{id}}
\newcommand{\bin}{\mathrm{bin}}
\newcommand{\unary}{\mathrm{unary}}
\newcommand{\val}{\mathrm{val}}

\newcommand{\Sym}{\mathrm{Sym}}

\newcommand{\init}{\mathrm{init}}

\newcommand{\WPs}{\mathrm{WPs}}

\title{Word problems and embedding-obstructions in cellular automata groups on groups}

\author{
Ville Salo \\
vosalo@utu.fi
}

\begin{document}
\maketitle

\begin{abstract}
We study groups of reversible cellular automata, or CA groups, on groups. More generally, we consider automorphism groups of subshifts of finite type on groups. It is known that word problems of CA groups on virtually nilpotent groups are in co-NP, and can be co-NP-hard. We show that under the Gap Conjecture of Grigorchuk, their word problems are PSPACE-hard on all other groups. On free and surface groups, we show that they are indeed always in PSPACE. On a group with co-NEXPTIME word problem, CA groups themselves have co-NEXPTIME word problem, and on the lamplighter group (which itself has polynomial-time word problem) we show they can be co-NEXPTIME-hard. We show also nonembeddability results: the group of cellular automata on a non-cyclic free group does not embed in the group of cellular automata on the integers (this solves a question of Barbieri, Carrasco-Vargas and Rivera-Burgos); and the group of cellular automata in dimension $D$ does not embed in a group of cellular automata in dimension $d$ if $D > d$ (this solves a question of Hochman).
\end{abstract}

\section{Introduction}

Reversible cellular automata are the continuous $G$-commuting bijections on $A^G$, for a finite alphabet $A$ and a countable group $G$. Equivalently, they are the bijections defined by a local rule (uniform over all $g \in G$), where we only look at cells in a finite neighborhood $gN$ of $g$ to determine the new symbol at $g$, using the local rule. Cellular automata are often seen as a model of parallel computation.

We study the following problem: If we fix a finite set $S$ of cellular automata, how hard is it to determine whether a given word $S^n$ gives a nontrivial cellular automaton, when we compose the generators in this order (as functions)? In other words, we study the complexity of the word problem of the group generated by $S$. For finite $S$ we call such groups \emph{groups of cellular automata}.

In symbolic dynamical terms, $A^G$ is a dynamical system called the full shift (on alphabet $A$ and group $G$), and the reversible cellular automata form its (topological) automorphism group. Many of our results are proved for more general sets (subshifts) $X \subset A^G$ in place of $A^G$.

One may summarize the present paper as follows: if the group $G$ contains subtrees that branch quickly enough (or cones that expand quickly enough), then we can propagate information along these subtrees quickly, and perform fast parallel computation. We obtain several technical statements (and in particular solve two open problems) with this idea, and we now proceed to describe these results.

\subsection{Complexity-theoretic results}

Appendix~\ref{sec:CT} gives the (mostly standard) complexity-theoretic definitions used in the statements.

Our first theorem is barely new, and is stated mainly for context, as it was shown essentially in \cite{Sa20e}. Namely \cite{Sa20e} proves the $\Z$-case, and only a short additional discussion is needed here (see Corollary~\ref{cor:Polynomial}). The co-NP algorithm is essentially from \cite{BoLiRu88,KiRo90}.

\begin{maintheorem}
\label{thm:A}
Let $G$ be an f.g.\ group with polynomial growth, and let $X \subset A^G$ be any subshift of finite type with language in NP. Then every finitely-generated subgroup of $\Aut(X)$ has word problem in co-NP. Furthermore, for any infinite f.g.\ group $G$ with polynomial growth, $\Aut(A^G)$ has a f.g.\ subgroup with NP-complete word problem.
\end{maintheorem}

For example $X = A^G$ satisfies the assumptions on $X$. Growth refers to the size of balls. Recall that polynomial growth corresponds to virtual nilpotency by Gromov's theorem \cite{Gr81a}, and that $O(n^D)$ growth implies $\Theta(n^d)$ growth for some $d$.

In particular, groups of cellular automata $H \leq \Aut(A^G)$ always have word problems in co-NP when the group has polynomial growth. Our first main result is about groups with faster growth, namely that if the growth of balls in a group is sufficiently superpolynomial, then the complexity ``jumps'' from co-NP to PSPACE. It is a standard conjecture in complexity theory that PSPACE is strictly larger than co-NP (equivalently, strictly larger than NP).

\begin{maintheorem}
\label{thm:B}
Let $G$ be a f.g.\ group with at least stretched exponential growth, and let $X \subset A^G$ be any SFT where some finite-support point has free shift orbit. Then $\Aut(X)$ has a finitely-generated subgroup with PSPACE-hard word problem.
\end{maintheorem}

For example $X = A^G$ satisfies the assumptions on $X$. Here, by having at least stretched exponential growth we mean that for some finite generating set (thus for any), the $k$-ball of $G$ has cardinality $\Omega(e^{k^\beta})$ for all $k$, for some \emph{degree}~$\beta$.

Under the Gap Conjecture of Grigorchuk \cite{Gr14}, the growth rates in the two theorems above are a dichotomy for f.g.\ groups. Specifically, the version $C^*(\beta)$ (for $\beta > 0$) of the Gap Conjecture states the following:

\begin{conjecture}
\label{con:Gap}
Every f.g.\ group that is not of polynomial growth has at least stretched exponential growth with degree $\beta$. 
\end{conjecture}

We immediately note that this is a parametrized family of conjectures, and the larger $\beta$ is, the more it is claiming. For large $\beta$, it claims too much: the conjecture is not true for e.g.\ $\beta > \log_{32} 31 = \gamma$, as the Grigorchuk group has growth $O(e^{n^\gamma})$ and is not virtually nilpotent \cite{Gr84} (much more is known about the growth of the Grigorchuk group \cite{ErZh20}). For our purposes, it would suffice that this conjecture is true for a single (possibly very small) $\beta > 0$.

Grigorchuk has several versions of the Gap Conjecture, and the best-known version is perhaps the one with $\beta = 1/2$, but where it is only assumed that $e^{\sqrt{n}}$ is not an upper bound for the growth, rather than that it is a lower bound. For our purposes, it is indeed important to have a stretched exponential lower bound.

Conjecture~\ref{con:Gap} is indeed known to hold in many cases (and for some $\beta > 0$): for $\beta = 1/6$ it holds for residually solvable and left-orderable groups \cite{Gr14}. It also holds rather trivially (for any $\beta < 1$) for many other types of groups, like hyperbolic groups, linear groups and elementary amenable groups. So among these groups, our result shows that the word problem of an automorphism group of a nice enough SFT (e.g.\ a full shift) is always either in coNP or is PSPACE-hard, and the jump happens exactly after polynomial growth. 

In several papers \cite{Sa22b,Sa20e,CaSa22,SaSc25}, the author (with various coauthors) has studied a specific subgroup of $\Aut(A^G)$ called $\PAut_G(A)$, which is generated by partial shifts (which split the aphabet into a Cartesian product and then shift the configuration on one of the resulting components), and symbol permutations (permuting the symbols in cells individually). This group is somewhat mysterious, but it provides a reasonably canonical f.g.\ subgroup, where many behaviors of the typically non-f.g.\ group $\Aut(A^G)$ can already be seen.

On a smaller class of groups, we show that this group can be used in the PSPACE-hardness proof. 

\begin{maintheorem}
\label{thm:C}
Let $G$ be a f.g.\ group containing a free submonoid on two generators, and let $A$ be any alphabet with sufficiently many divisors. Then $\PAut_G(A)$ has a PSPACE-hard word problem.
\end{maintheorem}


The main idea in all the theorems above is to build permutations of one part of the alphabet which only happen if some circuit evaluates to a true value on the other parts of the alphabet; or in terms of a probabilistic analog, ``a particular event happens''. For conditioning on polynomial-support events, this is known as Barrington's theorem after \cite{Ba89}, and it is basically the only ingredient needed in the proof Theorem~\ref{thm:A} in \cite{Sa20e}.

The main new idea needed is the use of fast information propagation along trees. Specifically, we can also work with events of exponential (or otherwise fast-growing) support, by using the parallelism of cellular automata. Unlike in Barrington's theorem, it seems we cannot do this while preserving a perfect separation of control and permuted data. The technique we use for this, and in fact in all our constructions, is something we call \emph{ripple catching}. 


For the general result Theorem~\ref{thm:B} (PSPACE-hardness on all groups with fast growth), we have to use some of combinatorics to show that all groups with stretched exponential growth contain large trees in an appropriate sense. This idea originates from a joint work in progress with Ilkka Törmä (though the computational model is different). We also need some symbolic dynamical arguments to construct markers on the nodes of these trees.

In a few cases we show the corresponding upper bound:

\begin{maintheorem}
\label{thm:D}
Let $G$ be a free group or a surface group, and let $X \subset A^G$ be any subshift of finite type with the finite extension property. Then every finitely generated subgroup of the automorphism group of $X$ has word problem in PSPACE.
\end{maintheorem}

For example $X = A^G$ satisfies the assumptions on $X$. We give more generally a sufficient geometric condition for such an upper bound in terms of ``polynomial splitting schemes'' for balls. 

Two remarks are in order. First, note that the word problem of $\PAut_G(A)$ contains the word problem of $G$, so of course if $G$ itself has PSPACE-hard word problem (or harder) then the PSPACE-hardness theorem is trivial. The word problem is, however, decidable in polynomial time in a large class of groups, and for example for a free group or a surface group it can be solved by Dehn's algorithm \cite{De11,Bo88a}.

Second, it is known that automata groups (whose definition we omit) can have PSPACE-complete word problems \cite{WaWe23} (and they always have PSPACE word problem \cite{St15}). Barrington's theorem is also a key component of the proof in \cite{WaWe23}.

For general groups, the natural upper bound for the complexity of the word problem of the automorphism group of a full shift is not PSPACE, but rather co-NEXPTIME. We show that at least one group (and thus every group containing it) reaches the upper bound.


\begin{maintheorem}
\label{thm:E}
On the lamplighter group $\Z_2 \wr \Z$, for every full shift on at least two letters, the automorphism group has a finitely-generated subgroup with co-NEXPTIME-complete word problem. For some alphabets $A$, we can use the group $\PAut(A)$.
\end{maintheorem}

The group $\Z_2 \wr \Z$ itself has word problem decidable in polynomial time. We believe the analog of Theorem~\ref{thm:E} can be proved similarly on solvable Baumslag-Solitar groups $\mathrm{BS}(1,n) = \Z[\frac1n] \rtimes \Z$, although we do not show this. 

The fact that we show the word problems to be this difficult for ``natural'' groups may be of interest in itself. There are certainly many general constructions that allow making the word problem of an f.g.\ group as difficult as one likes while preserving its decidability. A strong result in this direction is \cite{KhMySa17}.

Our results are of a different flavor in that they are simply a chapter in the study of properties of automorphism groups of subshifts, rather than constructed specifically for the purpose of hard word problems.

It seems that in most cases, when the complexity of the word problem of a group of general interest is known, it is known to be solvable in polynomial time. In fact, we only know the following cases where the word problem ended up being complete for a class conjecturally not solvable in polynomial time:
\begin{itemize}
\item the higher dimensional Brin-Thompson group $2V$ has co-NP-complete word problem \cite{Bi20};
\item on $\Z$, the groups $\PAut(A)$ have co-NP-complete word problems for suitable alphabets $A$ \cite{Sa20e};
\item automata groups can have PSPACE-complete word problems \cite{WaWe23}.
\end{itemize}

In the third item, the group was not a specific automata group that was studied before the paper, but rather one constructed for the proof. However, since one such group exists, one can likely find reasonably natural supergroups (generated, for example, by all automata that can be defined with a fixed number of states).

In the second item, the group explicitly used is $\PAut(A)$. One may wonder if this group is really natural: the generators are natural and study of the group emerged from questions unrelated to computability or complexity, but we do not even know, for example, whether it is decidable whether a given $f \in \Aut(A^\Z)$ belongs to this group. However (unlike in e.g.\ Theorem~\ref{thm:C} of the present paper), the proof in \cite{Sa20e} in fact only uses the subgroup of $R(A, B)$ of $\PAut(A \times B)$ where the $B$-track cannot affect the $A$-track in symbol permutations, which is a tamer and more natural group in many ways, and for example is elementarily amenable (see \cite{Sa22b}).

A final remark we make here is that in each of the three items above, it is only conjectural (even if widely-believed) that there is no polynomial-time algorithm for the word problem, as PSPACE $\neq$ P is not a proven fact. However, Theorem~\ref{thm:E} \emph{unconditionally} shows that there is no polynomial-time algorithm for the word problem of $\PAut(A)$ on the lamplighter group, and for some alphabets $A$.

\subsection{Non-embeddability results}

Besides complexity theoretic results, we obtain some non-embeddability results from the same construction.

These are based on an observation that we believe is entirely new (although in retrospect, rather obvious), namely that residual finiteness of $\Aut(A^{\Z^d})$ admits a quantitative variant when working with finitely-generated subgroups, based on the fact that a nontrivial automorphism with radius $n$ acts nontrivially on points with period $O(n)$ in each cardinal direction. See Lemma~\ref{lem:ResiduallyFinite}. This idea bears resemblance to the use of simple groups in \cite{Sc22,EpSc24,SaSc25}. 

The following was asked by Hochman in the conference ``Current Trends in Dynamical Systems and the Mathematical Legacy of Rufus Bowen'' organized in Vancouver in 2017. (A similar question, but allowing an arbitrary SFT in place of $\{0,1\}^{\Z^d}$ was asked in \cite{Ho10}.)

\begin{question}
Does $\Aut(\{0,1\}^{\Z^D})$ embed in $\Aut(\{0,1\}^{\Z^d})$ for $d < D$?
\end{question}

We prove the following result:

\begin{maintheorem}
\label{thm:F}
Suppose that $D > d$. Then $\Aut(A^{\Z^D})$ has a finitely-generated subgroup that does not embed in $\Aut(B^{\Z^d})$ for any nontrivial alphabets $A, B$.
\end{maintheorem}

There are no previously known nonembeddability results for the groups $\Aut(A^{\Z^d})$, indeed previously the only embeddability obstructions that have been observed or applied even in the case of $\Aut(A^\Z)$ are residual finiteness (i.e.\ all subgroups must be residually finite) and complexity of the word problem (i.e.\ all f.g.\ subgroups must have coNP word problem).

We note that from just $\Aut(A^{\Z^D})$ not embedding in $\Aut(B^{\Z^d})$, it would not be obvious that not all finitely-generated subgroups do not embed. For example, it is known that $\Aut(A^\Z)$ admits a f.g.\ subgroup that contains an isomorphic copy of every f.g.\ subgroup of $\Aut(A^\Z)$, yet it is not known whether there is an f.g.\ subgroup of $\Aut(A^\Z)$ that contains an embedded copy of $\Aut(A^\Z)$ \cite{Sa22b}.

As noted in \cite{BaCaRi25}, the groups $\Aut(A^{\Z^d})$ and $\Aut(B^{\Z^D})$ are non-isomorphic for $d \neq D$, since their centers are the respective shift groups $\Z^D$ and $\Z^d$. This is a result of Hochman \cite{Ho10} (generalizing Ryan's theorem \cite{Ry72}, and further generalized by Barbieri, Carrasco-Vargas and Rivera-Burgos in \cite{BaCaRi25}). However, it seems impossible to obtain any results about embeddability using this kind of method.

Hochman asked in personal communication whether Theorem~\ref{thm:F} is also true modulo the centers. Indeed, this is the case, by a minor modification of the proof.

\begin{maintheorem}
\label{thm:G}
Suppose that $D > d$, $\mathcal{G}_d = \Aut(A^{\Z^d}), \mathcal{G}_D = \Aut(B^{\Z^D})$ for any nontrivial alphabets $A, B$. Then $\mathcal{G}_D$ has a finitely-generated subgroup such that $\mathcal{G}_D/K$ does not embed in $\mathcal{G}_d/K'$ for any central subgroups $K \leq \mathcal{G}_D, K' \leq \mathcal{G}_d$.
\end{maintheorem}

In particular, the groups are not isomorphic modulo their centers.

Next, denote by $F_2$ the free group on two generators. The following is a question of Barbieri, Carrasco-Vargas and Rivera-Burgos \cite{BaCaRi25}:

\begin{question}
\label{q:F2Z}
Does $\Aut(\{0,1\}^{F_2})$ embed in $\Aut(\{0,1\}^\Z)$?
\end{question}

We solve this question as well.

\begin{maintheorem}
\label{thm:H}
Let $G$ be a f.g.\ group with at least stretched exponential growth. 
Then $\Aut(A^G)$ has a finitely generated subgroup that does not embed into $\Aut(B^{\Z^d})$ for any alphabet $B$ and $d \in \N$. 
\end{maintheorem} 

Note that this is an immediate corollary of Theorem~\ref{thm:A} and Theorem~\ref{thm:B} if PSPACE $\neq$ NP. However, our result is unconditional.

\section*{Acknowledgements}

We thank Ilkka Törmä for working through our first proof of Theorem~\ref{thm:HochmanSolution} in full detail. Of course, any mistakes are mine. The connection between the Gap Conjecture and a PSPACE frontier originated from the study of another computational model with Ilkka Törmä (apart from the construction of markers). We thank Sebasti\'an Barbieri for sharing Question~\ref{q:F2Z}, and Scott Schmieding for several inspiring discussions. We thank Laurent Bartholdi for helping to streamline the surface group argument (Lemma~\ref{lem:Surface}), and other useful comments. We thank Michael Hochman for asking about Theorem~\ref{thm:G}.

\section{Definitions, conventions and some standard observations}

In this section, we explain our definitions and conventions. In the appendix, we have included an index (Section~\ref{sec:Index}), which may help keep track of our notations.

Complexity-theoretic definitions are not found in this section, as they are quite lengthy, and mostly standard. We have instead written these in an appendix, which also contains some basic results on complexity theory. The reader may find it helpful to also consult a textbook such as \cite{ArBa09}.

Readers familiar with complexity theory will not need to read Section~\ref{sec:BasicCT} and Section~\ref{sec:AP} in the appendix, but may find it helpful to check out Lemma~\ref{lem:PSPACEByTrees} (an explicit tree formulation of a computation of an alternating Turing machine in terms of two nondeterministic Turing machines). We also give a proof of a version of Barrington's theorem in Section~\ref{sec:Barrington}. A reader familiar with previous work of the author may mostly skip Section~\ref{sec:Barrington}, while for others (even readers familiar with Barrington's theorem) we recommend taking a look, as we use similar notation and arguments in automorphism groups.

\subsection{Preliminaries}

First, our least standard definitions. To avoid excessive use of subscripts and superscripts, we will usually write $A|B$ instead of the more usual $A|_B$ in all of our notations. For the exponentiation notation $a^b$ in the context of conjugation, we use a caret symbol $\car$ and write $a \car b = b^{-1} \circ a \circ b$. We also set the convention that the $|$-symbol binds more tightly than $\car$, meaning $a|b\car c = (a|b)\car c$. Set theoretic operations like $A \cap B$ in turn bind more tightly than $|$, so for example $a|A \cap B \car c|D \cup E = (a|(A \cap B))\car (c|D \cup E)$. Composition $\circ$ binds the least tightly, after all of the above.

By $\bar{X}$ we mean the complement of the set $X$ when the containing set is clear from context. By $A \sqcup B$ we denote a union of two sets $A, B$, when $A \cap B = \emptyset$ (i.e.\ this is a domain-restriction of the union operation). In a cartesian product $A \times B$, we refer to $A$ and $B$ as \emph{components}, but especially in the context of sets $(A \times B)^G$, we refer to $A$ and $B$, as well as the projections $A^G$ and $B^G$, as \emph{tracks}. Intervals $[i, j]$ are interpreted on $\Z$, i.e.\ $[i, j] = \{k \in \Z \;|\; i \leq k \leq j\}$, unless explicitly specified otherwise. By $\Z_k$ we denote the cyclic group with $k$ elements. By $1_G$ we denote the identity element of a group $G$.

In many cases, we will have either a graph $G$ or a configuration $x$, and on its nodes $u \in G$ or positions $x_g$, we have colors (or symbols). We often think of the graph or configuration as being fixed, and speak of the color at $u$ or $g$. The color in turn is often a Cartesian product of the form $C \times B$. Then \emph{$B$-value} of a node $u$ or element $g$ refers to the $B$-component of $x_g$ or of the color of $u$, and similarly for $C$.

For a function $g : \N \to \R$ write $O(g)$ (resp.\ $\Omega(g)$) for functions $f$ such that for some $C > 0$, $f(n) \leq Cg(n)$ (resp.\ $f(n) \geq Cg(n)$) for large enough $n$. Write $\Theta(g)$ for the intersection. It is customary to use the notation $f = O(g), f = \Omega(g), f = \Theta(g)$ instead of inclusion.

For nondecreasing functions $f, g : R \to \R$ with $R \subset \R$, write $f \lesssim g$ if there exist $a, b > 0$ such that for some $r_0$, for all $r_0 \leq r \in R$ we have $br \in R$ and $f(r) \leq ag(br)$. Write $f \gtrsim g \iff g \lesssim f$. If $f \lesssim g$ and $f \gtrsim g$, we write $f \asymp g$ and say that $f, g$ are \emph{asymptotically equal}. Note that all functions in $f(\Theta(n))$ are asymptotically equal. The term \emph{growth} refers to a class of functions closed under asymptotic equality.

\begin{lemma}
\label{lem:Invert}
If $f(n \log^d n) \lesssim g(n)$, then $f(n) \lesssim g(n/\log^d n)$. Similarly for $\gtrsim$.
\end{lemma}

\begin{proof}
Setting $m = n / \log^d n$ we have
\[ m \log^d m = (n / \log^d n) \log^d (n / \log^d n) = n - n \log^d \log^d n / \log^d n = \Theta(n) \]
Thus if $f(n \log^d n) \lesssim g(n)$, we have
\[ f(n) \asymp f(\Theta(m \log^d m)) \lesssim g(m) = g(n/\log^d n). \]
The case of $\gtrsim$ is similar.
\end{proof}

In the context of permutation groups on some set (clear from context), we write the permutation $\pi$ mapping $\pi(a) = b, \pi(b) = c, \pi(c) = a$ and fixing all other elements as $(a \; b \; c)$ (or $(a; \; b;\; c)$ if $a, b, c$ are longer expressions). Such a permutation is called a \emph{$3$-rotation}.

An \emph{alphabet} is a finite set of \emph{symbols}. A \emph{word} $u$ is an element of the free monoid $A^*$ on an alphabet $A$, and its \emph{length} is the number of $A$-symbols (free generators) it is composed of. We by default \emph{$0$-index} words meaning $u = u_0u_1\ldots u_{|u|-1}$, but this can be overridden by notation, or if we explicitly say we \emph{$1$-index} the word i.e.\ $u = u_1u_2\ldots u_{|u|}$. We write $A^{\leq n}$ for the words of length at most $n$.

For words $u, v$ we denote the free monoid operation, i.e.\ concatenation, by $uv$. The \emph{empty word} (unique word of length $0$) is written $\epsilon$, and we write $A^+ = A^* \setminus \{\epsilon\}$. A set of words $W \subset A^*$ is \emph{mutually unbordered} if, whenever $u, v \in W$, $i, j \in \Z$, and $x \in A^\Z$ satisfies $x|[i, i+|u|-1] = u$, $x|[j, j+|v|-1] = v$, we have $[i, i+|u|-1] \cap [j, j+|v|-1] = \emptyset$, unless $i = j$ and $u = v$. A word $w$ is \emph{unbordered} if $\{w\}$ is mutually unbordered.

We use the usual notations $O(f(n)), \Theta(f(n)), \Omega(f(n))$: $f'(n) = O(f(n))$ if $f'(n) \leq Cf(n)$ for large enough $n$ and some $C > 0$; $f'(n) = \Omega(f(n))$ if $f'(n) \geq Cf(n)$ for large enough $n$ and some $C > 0$; $f'(n) = \Theta(f(n))$ if both are true.

Vertices of a graph of any kind are written $V(G)$ and edges as $E(G)$. We consider undirected graphs whose edges are sets $\{u, v\}$ of cardinality $2$ with $u, v \in V(G)$, as well as directed graphs (defined later).

In a metric space $X$, we say $Y \subset X$ is \emph{$r$-dense} if every $r$-ball intersects $Y$. We say $Y$ is $r$-separated if $d(y, y') < r$ for $y, y' \in Y$ implies $y = y'$. The following is immediate in finite metric spaces (in general, axiom of choice is needed):

\begin{lemma}[Packing Lemma]
In any metric space $X$, for any $r$ there exists an $r$-dense $r$-separated subset $Y \subset X$.
\end{lemma}

\begin{proof}
If $Y$ is maximal $r$-separated, then we cannot add any new points, so for every $x \in X$, $\{x\} \cup Y$ is not $r$-separated, i.e.\ $d(x, y) < r$ for some $y \in Y$. We conclude that any such maximal $Y$ is $r$-dense.
\end{proof}

\subsection{Groups}

Our groups $G$ always have a fixed generating set $S$ which is \emph{symmetric} $S = \{s^{-1} \;|\; s \in S\}$ and $1_G \in G$, and $G$ is considered as a metric space with the corresponding word metric
\[ d(g, h) = \min\{n \;|\; h = gs_1s_2\ldots s_n \mbox{ for some } s_i \in S\}, \]
and \emph{($k$-)balls} are $B_k = B_k^G = \{g\in G \;|\; d(1_G, G) \leq n\}$ or their translates $gB_k$. A \emph{geodesic} in a group is a sequence of group elements $g_1, \ldots, g_n$ such that $d(g_i, g_j) = |i-j|$ for $i, j$. It is obvious that geodesics of arbitrary length exist in f.g.\ groups.

The \emph{commutator} of group elements $a, b$ is $[a, b] = a^{-1}b{-1}ab$; $[G, G]$ is the group generated by commutators $[a, b]$ where $a, b \in G$. Also define longer commutators by $[a_1, a_2, \ldots, a_k] = [a_1, [a_2, \ldots, a_k]]$. A group is \emph{abelian} if any two elements $g, h$ commute. A group $G$ is \emph{nilpotent} if the \emph{lower central series} $G, G_1 = [G, G], G_2 = [G, G_1], G_3 = [G, G_2], \ldots$ reaches the trivial group in finitely many steps. A group $G$ is \emph{solvable} if the \emph{derived series} $G, G^1 = [G, G], G^2 = [G^1, G^1], \ldots$ reaches the trivial group in finitely many steps. The \emph{center} of a group $G$ is the set of elements $g \in G$ that commute with all other elements.

Two groups $G,H$ are \emph{isomorphic} if there is a bijection between them that preserves the multiplication relation, and we write this as $G \cong H$. The groups $G, H$ are \emph{commensurable} if there exist subgroups $L \leq G, L' \leq H$ of finite index such that $L \cong L'$. The \emph{index} $[G : H]$ of a subgroup $H \leq G$ is the cardinality of $G/H = \{gH \;|\; g \in G\}$.

A group $G$ is \emph{virtually P} (where P is any group, property of groups or family of groups) if $G$ has a subgroup of finite index which is P. It is \emph{residually P} if for all $g \in G$ there is a quotient in $P$ where $g$ maps nontrivially. A group $G$ is \emph{amenable} if for all $\epsilon>0$ and $S \subset G$ finite, there exists finite $F \subset G$ such that $\frac{FS}{F} < 1+\epsilon$. A group $G$ is \emph{left-orderable} if there exists a total order $<$ on $G$ such that $a < b \implies ga < gb$ for all $a, b, g \in G$.

The \emph{word problem} of a group $G$, with respect to a finite generating set $S$, is the language $\{w \in S^* \;|\; w =_G 1_G\}$, where $=_G$ denotes equality when evaluated in the group $G$.

The symmetric group (of all permutations, under composition) on a finite set $A$ is $\Sym(A)$, and $\Alt(A)$ is the simple index-$2$ subgroup of even permutations. For integer $n$, we write $\Sym(n), \Alt(n)$ for $\Sym([1,n]), \Alt([1,n])$ respectively.

\subsection{Symbolic dynamics}

In the context of a group $G$ and alphabet $A$, a \emph{pattern} is $p \in A^D$. If $q \in A^E$ for some $E \subset G$, $q$ \emph{contains a translate of} $p$ if there exists $g \in G$ such that $\forall h \in D: q_{gh} = p_h$.

The set $A^G$ for finite alphabet $A$ and countable group $G$ is given its product topology (making it a Cantor space), and it has a natural continuous action of $G$ by $gx_h = x_{g^{-1}h}$ for $g, h \in G, x \in A^G$. The \emph{full shift} is $A^G$ as a dynamical system under the $G$-action. A \emph{subshift} is a $G$-invariant closed subset of $A^G$ (i.e.\ a subsystem of a full shift). A \emph{subshift of finite type} or \emph{SFT} is $X \subset A^G$, such that for some clopen set $C \subset A^G$, we have $X = \{x \in A^G \;|\; \forall g \in G: gx \notin C\}$ (such a set is indeed a subshift).

A subshift has the \emph{finite extension property}, or is \emph{FEP} if there exists a finite set $D \subset G$, a set of patterns $P \subset A^D$ and a \emph{look-ahead} $r\in\N$ such that for all $E \subset G$ and $x \in A^E$, we have $x \in X|E$ (i.e.\ $x$ is \emph{globally valid}) if and only if there exists $y \in A^{EB_r}$ such that $y$ does not contain any translate of a pattern from $P$ (i.e.\ $y$ is \emph{locally valid}). An FEP subshift is always an SFT. 
Generalizing FEP, we say a subshift has the \emph{polynomial extension property for balls} or \emph{PEPB} if, with respect to some finite set of forbidden patterns $P$, and for some \emph{degree} $t$, a pattern $p \in A^{B_n}$ is globally valid if and only if some extension $q \in A^{B_{n^t}}$ is locally valid. 

The elements $x \in A^G$ are called \emph{configurations} or \emph{points}. The \emph{support} of a configuration is $\{g \in G\;|\; x_g \neq 0\}$, where $0 \in A$ is some agreed-upon zero-symbol (usually precisely $0$).

An \emph{endomorphism} of a subshift $X$ is a $G$-equivariant continuous $\phi : X \to X$. If there is a two-sided inverse (which is also an endomorphism), $\phi$ is an \emph{automorphism}. In the case $X = A^G$, endomorphisms are called \emph{cellular automata} and automorphisms \emph{reversible cellular automata}. An endomorphism always has a \emph{radius} $r \in\N$ and a \emph{local rule} $f_{\mathrm{loc}} : A^{B_r} \to A$ such that $f(x)_g = f_{\mathrm{loc}}(x|gB_r)$, where we see $x|gB_r$ as an element of $A^{B_r}$ in the obvious way.

Two subshifts $X \subset A^G, Y \subset B^G$ are \emph{(topologically) conjugate} if there is a homeomorphism $\phi : X \to Y$ that commutes with the $G$-actions; $\phi$ is called a \emph{conjugacy} and it \emph{conjugates} $X$ with $Y$. Generalizing endomorphisms and conjugacies, a continuous shift-commuting map $\phi : X \to Y$ is called a \emph{block map}. When such a map is injective, it conjugates $X$ with its image. The SFT property is known to be preserved under conjugacy. Also block maps (and thus conjugacies) have local rules, defined in an analogous way as for cellular automata.


\subsection{$\PAut$}

We define a reasonably natural (or at least canonical) f.g.\ subgroup of the automorphism group that we use in some of the results.

\begin{definition}
Let $G$ be a group, and let $A_1, ..., A_k$ be finite sets and $A = \prod_i A_i$. Then $\PAut_G(A_1, ..., A_k)$ is the group generated by
\begin{itemize}
\item partial shifts $\sigma_{g, i}$ for $g \in G, i \in [1, k]$, which on input $x = (x_1, \ldots, x_k) \in A^G$ map $\sigma_{g, i}(x) = (x_1, \ldots, x_{i-1}, y, x_{i+1}, \ldots, x_k)$ where $y_h = x_{hg}$ for $h \in G$; and
\item symbol permutations $\eta_\pi$ for $\pi \in \Sym(A)$, which on input $x \in A^G$ are defined by $\eta_\pi(x)_h = \pi(x_h)$ for $h \in G$.
\end{itemize}
For any finite alphabet $A$, define $\PAut(A)$ to be $\PAut(A_1, ..., A_k)$ where $A \cong A_1 \times \cdots \times A_k$ through any bijection, where $|A_i|$ is prime.
\end{definition}

When $G$ is clear from context, we write just $\PAut(A) = \PAut_G(A)$, or even just $\PAut$ if also the alphabet is clear from context.

\begin{lemma}
\label{lem:PAutWD}
$\PAut(A)$ is well-defined.
\end{lemma}

\begin{proof}
This is entirely analogous to \cite[Lemma~2.1]{Sa22b}. The point is the change between the groups for different choices is in the partial shifts. Since the prime decomposition is unique, these are conjugated by a symbol permutation, which is in the group in any case.
\end{proof}


Suppose $A = C \times B$, $G$ is a finitely-generated group, $\pi \in \Sym(C)$, $c \in C$ is not in the support of $\pi$, and $s \in G \setminus \{1_G\}$. Then we define $\phi_{\pi, c, s}$ as the automorphism of $A^G$ that applies $\pi$ to the $C$-component of the symbol at $g$ if the $C$-symbol at $gs$ is equal to $c$.

\begin{lemma}
\label{lem:PhisInPAut}
Let $|C| \geq 4, |B| \geq 2$ be alphabets and $\pi \in \Alt(C)$. Suppose $c$ is not in the support of $\pi$. Then $\phi_{\pi, c, s} \in \PAut_G(B \times C)$.
\end{lemma}

\begin{proof}
Assume $B = \Z_h$ for some $h \geq 2$. Use Ore's theorem to obtain for each $\pi \in \Alt(C)$ a decomposition $\pi = [\pi_0, \pi_1]$ where neither $\pi_i$ has $c$ in its support. We define automorphisms by describing their effect on the symbol $x_g$ of the input $x \in A^G$, for $g \in G$. Define an automorphism $\theta_b$ that, if the $B$-value at $x_{gs}$ is $b$, permutes the first component of $x_g$ by $\pi_b$. Clearly this is in $\PAut$. Next define an automorphism $\theta$ that permutes the $B$-component of $x_g$ by $(0 \; 1)$ if the $C$-component of $x_g$ is $c$. This is conjugate to a symbol permutation by a partial shift, so again in $\PAut$. Finally, let $\theta'$ be the automorphism that permutes the $B$-component of $x_g$ by a full length $|B|$ rotation in any case.

We claim that
\begin{equation}
\label{eq:eq}
\phi_{\pi, c, s} = \prod_{i = 0}^{|B|-1} [\theta_0, \theta_1 \car \theta] \car (\theta')^i
\end{equation}
To see this, we first analyze $\bar\theta = [\theta_0, \theta_1 \car \theta]]$. Observe that permutations applied to the $C$-component do not affect which cells contain $c$ since $c$ is not in the support of the $\pi_i$, so $\theta$ affects the $B$-component of a particular cell $g$ either in every application, or not at all.

In particular, in $\theta_1 \car \theta$ the possible changes from applying $\theta$ of the $B$-values are temporary, and are only visible when applying $\theta_1$. It follows that each of $\theta_0$, $\theta_1$ applies at some cell $g$ either both times it appears in the commutator (i.e.\ both the positive and negative application), or neither time. If one of them does not apply at $g$, then $\bar\theta$ has no action at $g$. If the conditions do hold, then precisely $\pi = [\pi_0, \pi_1]$ is applied to the $C$-value.

We see that the conditions for $\theta_0$ applying is precisely that the $B$-value at $gs$ is $0$. If this is the case, then $\theta_1 \car \theta$ applies if and only if the $C$-value at $gs$ is $c$. We conclude that $\bar\theta$ applies $\pi$ to the $C$-color at $g$ if and only if $gs$ has $B$-value $0$ and $C$-value $c$. Composing with all $(\theta')^i$-conjugates, we instead check that the $B$-value is $-i$. Thus, the composition only checks that the $C$-value is $c$, as then exactly one of the conjugates applies. This proves Equation~\ref{eq:eq}.
\end{proof}

For finite $S \Subset G$, define $\phi_{\pi, c, S}$ as the automorphism of $A^G$ that applies $\pi$ to the $C$-component of the symbol at $g$ if the $C$-symbol at $gs$ is equal to $c$ for all $s \in S$.

\begin{lemma}
\label{lem:PhiInPAut}
Let $S \Subset G$ be finite, let $|C| \geq 6, |B| \geq 2$ be alphabets and $\pi \in \Alt(C)$. Suppose $c$ is not in the support of $\pi$. Then $\phi_{\pi, c, S} \in \PAut_G(B \times C)$.
\end{lemma}

\begin{proof}
We may write $S = \{s_1, s_2, \ldots, s_\ell\}$. Since $|C| \geq 6$, we may write $\pi = [\pi_1, \pi_2, \ldots, \pi_\ell]$ where none of the $\pi_i \in \Alt(C)$ have $c$ in their support. Then 
\[ \phi_{\pi, c, S} = [\phi_{\pi_1, c, S}, \phi_{\pi_2, c, S}, \ldots, \phi_{\pi_\ell, c, S}] \]
since the condition for some $\phi_{\pi_i, c, S}$ (equivalently, all of them) applying the permutation is only a function of which nodes the symbol $c$ appears as the $C$-value, which is not affected by the permutations $\pi_i$.
\end{proof}

\section{Ripple catching}
\label{sec:RippleCatching}

In this section, we present the main new technology of the paper, namely ripple catching. This is a simple trick that allows us to use parallelism to ``check'' that a large cone exists at a position, in the positive direction in a graph, and that certain successor relations (which can be anything checkable in NC$^1$) are respected. Specifically, a particular permutation is applied at a node if and only if the cone rooted at that element exists and everywhere has locally valid content, in some desired sense.

To realize this, we consider a ``ripple of permutations'' down from the rim of the cone to its root, and we condition a permutation at the root of the cone on the ripple reaching the it. Using a commutator of two such ripples, we can ensure that only a ``complete'' ripple results in a nontrivial permutation.

Specifically, we will consider a graph whose nodes have colors in $C \times B$, and we will perform a permutation on the $C$-component, controlled by an essentially arbitrary condition on the colors $B$ (defined by some successor relation), and also the condition that $C$-colors are all $0$ outside the root of the cone.

\begin{definition}
Let $S$ be a finite set (one may think of it as the positive generating set of a free group). An \emph{$S$-labeled $C$-colored graph} is a directed graph $G$ where each edge has a label from $S$ or $S^{-1}$, and we write $u \overset{s}\rightarrow v$ if there is an $s$-labeled edge from $u$ to $v$; and each vertex $u$ has a color in $C$. If $C$ is irrelevant (or taken to be a singleton set) we speak of simply $S$-labeled graphs. We say an $S$-labeled graph is \emph{good} if
\begin{itemize}
\item whenever $u \overset{s}\rightarrow v$ is an edge, there is also an edge $v \overset{s^{-1}}\rightarrow u$;
\item for each $u$ and $s \in S$, there is a most one $v$ with $u \overset{s}\rightarrow v$.
\end{itemize}
\end{definition}

In the present section, we will not need $s$-edges for $s \in S^{-1}$, but we need them in Section~\ref{sec:Arboreous} so we include them in this definition.

If an $S$-labeled graph is good and $u \overset{s}\rightarrow v$, we call $v$ the \emph{$s$-successor} of $u$, and we write it as $us$. The $s$-successor of a node either is unique, or does not exist. Define inductively the $w$-successor of $u$, where $w = st \in S^+$ (with $s \in S$, $t \in S^*$) by $u(st) = (us)t$, again if it exists. A good $S$-labeled graph can be thought of as a partial group action of the free group $F_S$. We also write $us$ for the unique neighbor of $u$ where the edge has label $s$ (for $s \in S^\pm$), if this exists, and otherwise $us$ is undefined. 

Let $C$ and $S$ be finite sets, let $B$ be a countable set, and suppose we have associated to each $s \in S$ a \emph{successor relation} $R_s \subset B \times B$. If $(b, b') \in R_s$, we say $b'$ is an \emph{$s$-successor} of $b$. Often, the successor relation is deterministic, in that there is at most one $s$-successor for each color $b \in B$. In this case, we can also describe the successor relations by giving a partial action of $S$ (and thus $S^*$).

Further suppose $B$ is \emph{ranked} meaning there is a \emph{rank function} $\rho : B \to \N \cup \{\bot\}$ such that $\rho(b') = \rho(b) + 1$ or $\rho(b') = \bot$ for all $b, b' \in B$ such that $b'$ is an $s$-successor of $b$ for some $s \in S$. Consider the set $\mathcal{Q}_{S, C, B}$ of all $S$-labeled $(C \times B)$-colored good graphs on a fixed vertex set.\footnote{The fixed vertex set is just for set-theoretic reasons, we may assume it is large enough for all practical purposes. Since the group is countable, there is in any case a faithful action on a countable subset.}

Say a node $u$ in a graph is \emph{successful} if its color is $(c, b)$, and for each $s \in S$, its $s$-successor $v$ exists, and the $B$-component of the color of $v$ is an $s$-successor of $b$ for each $s$.

\begin{definition}
Let $S, C, B$ be finite sets, with $|C| \geq 6$. Define the \emph{ripple group} $\mathcal{G}_{S, C, B}$ by the following generators acting on $\mathcal{Q}_{S, C, B}$:
\begin{itemize}
\item If $\pi \in \Alt(C)$, $s \in S$, and $c \in C$ is not in the support of $\pi$, the generator $\phi_{\pi,c}$ applies $\pi$ on the $C$-component $u$ if whenever $us$ exists, the $C$-color of $us$ is $c$.
\item If $\pi \in \Alt(C)$, the generator $\gamma_\pi$ applies $\pi$ on the $C$-component of the color of a node $u$, if $u$ is successful.
\item If $\pi \in \Alt(C)$ and $n \in \N$, the generator $\beta_{\pi, \ell}$ applies $\pi$ on the $C$-component of the color of a node $u$ if the $B$-component of $u$ contains $b$ such that $\rho(b) = \ell$. 
\end{itemize}
For each $n$, the group $\mathcal{G}_{S, C, B}$ is endowed with the left-invariant word $n$-metric given by these generators, with $\ell$ bounded from above by $n$. 
\end{definition}

Some simple remarks are in order:
\begin{enumerate}
\item None of the generators affect the shape of the graph or the $B$-components of colors, only the $C$-components.
\item All of $\phi_{\pi, c}, \gamma_{\pi}, \beta_{\pi, \ell}$ obviously define bijections: they perform permutations on $C$-components of some nodes, and the conditions under which these permutations are not affected by the automorphism itself, so they have finite order.
\item The name of $\phi_{\pi,c}$ is very close to $\phi_{\pi,c,s}$ and $\phi_{\pi,c,S}$ from the previous section, but there is no direct clash, and $\phi_{\pi,c}$ can be seen as a generalization (or abstraction) of $\phi_{\pi,c,S}$.
\item The word $n$-metric is proper for all $n$ in the sense that balls are finite, but since these are not a generating set for the entire group $\mathcal{G}_{S, C, B}$, distances can be infinite (though one can alternatively consider the metric in the group generated by the finitely many $n$-bounded generators).
\end{enumerate}

We now define an important element of this group, which ``combines'' the three types of generators into one.

\begin{definition}
Let $S, C, B$ be finite sets, with $|C| \geq 6$. If $\pi \in \Alt(C)$, then the map $\psi_{\pi, \ell, c}$ acts on $\mathcal{Q}_{S, C, B}$ as follows: it applies $\pi$ on the $C$-component of a node $u$ if
\begin{itemize}
\item $u$ is successful,
\item the $C$-component of the color of $us$ is $c$ for all $s \in S$, and
\item $\rho(b) = \ell$, where $b$ is the $B$-value of $u$.
\end{itemize}
\end{definition}


\begin{lemma}
Let $S, C, B$ be finite sets with $|C| \geq 6$. If $\pi \in \Alt(C)$ and $c$ is in the support of $\pi$, then $\psi_{\pi, \ell, c}$ is a permutation, and $\psi_{\pi, \ell, c} \in \mathcal{G}_{S, C, B}$.
\end{lemma}

\begin{proof}
Suppose first that $c$ is not in the support of $\pi$. Then it is easy to see that again the set of nodes where the $C$-value is changed is invariant under the action, and again the action of $\psi_{\pi, \ell, c}$ has finite order, thus is a permutation. Namely, $\psi_{\pi, \ell, c}$ does not change the nodes where the $C$-value is $c$ whether or not $\pi$ is applied (since $c$ is not in the support of $\pi$), and the other two conditions speak only about $B$-values, which are never affected.


Since $|C| \geq 6$, we can write any permutation $\pi$ without $c$ in its support as $\pi = [\pi_1, \pi_2, \pi_3, \pi_4]$. We now show that
\[ \psi_{\pi, \ell, c} = [\phi_{\pi_1,s,c}, \gamma_{\pi_2}, \beta_{\pi_3, \ell}]. \]

For this, observe that none of the maps $\phi_{\pi_1,s,c}, \gamma_{\pi_2}, \beta_{\pi_3, \ell}$ affect the set of cells where the others act: The only dependence of any of them on $C$-values is that whether $\phi_{\pi_1,s,c}$ applies the permutation $\pi_1$ depends on the set of nodes with $C$-value $c$. Since $c$ is not in the support of any permutation, so this set never changes.

This means that in any particular cell, some of the $\phi_{\pi_1,s,c}, \gamma_{\pi_2}, \beta_{\pi_3, \ell}$ apply whenever they appear in the commutator (with positive or negative power), and some never apply. If even one of them never applies, then the commutator cancels. If all of them always apply, then the permutation applied on the $C$-value is $[\pi_1, \pi_2, \pi_3] = \pi$.

This happens if and only if the conditions of all three maps hold, i.e.\ precisely if 
\begin{itemize}
\item each existing $us$ for $s \in S$ has $C$-value $c$ (the condition for $\phi_{\pi_1,c}$);
\item $u$ is successful (the condition for $\gamma_{\pi_2}$);
\item $\rho(u) = \ell$ (the condition for $\beta_{\pi_3, \ell}$).
\end{itemize}
This is precisely the definition of $\psi_{\pi, \ell, c}$.

Next, suppose $c$ is in the support of $\pi$. We again first explain why $\psi_{\pi, \ell, c}$ is at least a bijection. Since we only consider ranked graphs, if node $u$ is successful and has color $(c, b)$ with $\rho(b) = \ell$, then $us$ has $B$-component $bs$ in its color, with $\rho(bs) = \ell+1$. Therefore, the $C$-component of this node is not affected by $\psi_{\pi, \ell, c}$, and it is easy to see that $\psi_{\pi, \ell, c}$ is again of finite order, thus a permutation.

To get this in the group, we write $\pi$ as a product of $3$-cycles, and observe that $\psi_{\pi, \ell, c}$ can be written as the corresponding product (since the set of cells where the permutation applies in the $C$-component is never changed, by the previous paragraph). Thus, we may assume $\pi$ is already a $3$-cycle. Since $|C| \geq 4$, we can take $\pi' \in \Alt(C)$ such that the support of $\pi \car (\pi')^{-1}$ does not contain $c$.

We have $\psi_{\pi \car (\pi')^{-1}, \ell, c} \in \mathcal{G}_{S, C, B}$ by the first half of the proof. It is now easy to verify that
\[ \psi_{\pi, \ell, c} = \psi_{\pi \car (\pi')^{-1}, \ell, c} \car \gamma_{\pi', \ell}. \]
Namely, if $\rho(u) \neq \ell$, then clearly neither automorphism acts on the $C$-component of the color at $u$.

If $\rho(u) = \ell$, then if $u$ is not successful, neither side acts. If $u$ is successful and has color $(c, b)$, then each $us$ exists, and the $B$-component $b'$ of the color of $us$ is an $s$-successor of $b$, thus has rank $\rho(b)+1$. Thus, the $C$-component of $us$ is not affected by the element $\psi_{\pi \car (\pi')^{-1}, \ell, c} \car \gamma_{\pi', \ell}$. If the $C$-component of the color of each $us$ is $0$, then $\psi_{\pi, \ell, c}$ applies $\pi$ and $\psi_{\pi \car (\pi')^{-1}, \ell, c} \car \gamma_{\pi', \ell}$ applies $(\pi \car (\pi')^{-1}) \car \pi' = \pi$; in other cases, neither automorphism affects the $C$-component of the color at $u$.
\end{proof}

\begin{definition}
The \emph{cone} $G(u)$ at the node $u$ in a graph $G \in \mathcal{Q}_{S, C, D}$ is the smallest set $N$ of such that $u \in N$, and if $v \in N$ is successful, then $vS \subset N$. The node $u$ is the \emph{root} of the cone. Write $G_i(u) = G(u) \cap uS^i$. The \emph{cone of depth $n$} $G_{\leq n}(u)$ is $\bigcup_{j = 0}^n G_j(u)$. We say the cone is \emph{full} if $G_{\leq n}(u) = uS^{\leq n}$.
\end{definition}

Note that the nodes in $G_i(u)$ need not themselves be successful.


\begin{lemma}[Ripple catching lemma]
\label{lem:RippleCatching}
Consider the group $\mathcal{G}_{S, C, B}$ with $B, C$ finite sets with $|C| \geq 4$. Let $\pi \in \Alt(C)$. Then there is an element in $\mathcal{G}_{S, C, B}$ of word $n$-norm linear in $n$, whose behavior can be described as follows: in the $C$-component of a node $u$, we apply $\pi$ if and only if there is a full $n$-cone at $u$, 
 and for all $u \neq v \in G_{\leq n}(u)$, the $C$-value is $0$.
\end{lemma}

\begin{proof}
We may assume $0, 1, 2 \in C$. Let $\pi = (0 \; 1 \; 2)$, and define
\[ \psi_1 = \left(\prod_{\ell = 1}^{n-1} \psi_{\pi, \ell, 1} \right) \circ \gamma_{\pi, n} \]
and
\[ \psi_2 = \left(\prod_{\ell = 1}^{n-1} \psi_{\pi^{-1}, \ell, 2} \right) \circ \gamma_{\pi^{-1}, n}, \]
(where we recall that the product puts the $i = 0$ value first, so $\psi_{\pi, 1, 0}$ is applied last when we act from the left).

Now, let $\pi_1, \pi_2$ be arbitrary permutations such that $[\pi_1, \pi_2] \neq 1$, and consider
\[ \psi = [\psi_{\pi_2, 0, 2} \car \psi_2, \psi_{\pi_1, 0, 1} \car \psi_1]. \]
We observe first that $\psi_{\pi_2, 0, 1}$ and $\psi_{\pi_2, 0, 1}$ only affect the $C$-values of nodes $u$ where the $B$-value $b$ satisfies $\rho(b) = 0$.

On the other hand, if $\ell, n \geq 1$, the set of nodes $v$ where $\psi_{\pi, \ell, 1}$, $\gamma_{\pi, n}$, $\psi_{\pi^{-1}, \ell, 2}$ or $\gamma_{\pi^{-1}, n}$ applies does not depend on the $C$-values of such nodes $u$. Thus, in the actions of $\psi_{\pi_2, 0, 2}^{\psi_1}$ and $\psi_{\pi_1, 0, 1} \car \psi_1$, the actions of $\psi_2$ and $\psi_1$ cancel, and thus the only possible nodes where $\psi$ may have an effect are those where $\rho(b) = 0$ for the $B$-value $b$, and the only possible effect is that $[\pi_2, \pi_1]$ is applied to the $C$-value.

Consider now a successful node $u$ where the $B$-value $b$ satisfies $\rho(b) = 0$. We claim that $[\pi_2, \pi_1]$ is applied to the $C$-value if and only if all nodes $ut$ exist and are successful for $t \in S^\ell$ for $0 \leq \ell < n$, and the $C$-value is $0$ in all nodes $ut$ with $t \in S^{\ell}$ for $0 < \ell < n$.

Note immediately that both $\psi_{\pi_2, 0, 2} \car \psi_2$ and $\psi_{\pi_1, 0, 1} \car \psi_1$ being applied at $u$ means precisely that the $C$-value at each $us$ with $s \in S$ is equal to $2$ after applying $\psi_2$, and equal to $1$ after applying $\psi_1$.


For $j \in \{1, \ldots, n-2\}$, we say $v \in G_j(u)$ is \emph{active} in \emph{ripple} $p \in \{1, 2\}$ if after applying
\[ (\prod_{\ell = 1}^{i-1} \psi_{\pi, \ell, p})^{-1} \psi_p \]
(i.e.\ just after applying the first action that affects the nodes at the level of $v$), the $C$-value at $v$ is $p$. Note that $\psi$ acts nontrivially at $u$ if and only if all nodes in $uS$ are active in both ripples (note that $uS = G_1(u)$ by the assumption that $u$ is successful).

We now observe that if $v \in G_{j-1}(u)$ and $vs$ is not active in ripple $p$, then if $j > 1$, the node $v$ can not be active in both of the ripples $p, 3-p$. Namely, suppose $v$ is active in ripple $1$. Then necessarily we have either that
\begin{itemize}
\item $v$ has $C$-value $0$ and $vs$ is active in ripple $1$, or
\item $v$ has $C$-value $1$ and $vs$ is inactive in ripple $1$.
\end{itemize}
Similarly, if $v$ is active in ripple $2$, then either
\begin{itemize}
\item $v$ has $C$-value $0$ and $vs$ is active in ripple $2$, or
\item $v$ has $C$-value $2$ and $vs$ is inactive in ripple $2$.
\end{itemize}

If $v$ has $C$-value $0$, then since $vs$ is not active in both ripples, $v$ will not be either. If the $C$-value of $v$ is $1$ or $2$ originally, then it has the same value when either ripple reaches it, so it cannot be $1$ during one ripple and $2$ in the other.

We conclude that the only way $\psi$ can modify the $C$-value at $u$ is indeed that all $v \in G_j(x)$ are active for $1 \leq j \leq n-1$, which means precisely that $u$ starts a full valid cone full of $0$s.
\end{proof}

As already suggested in the beginning of the section, when we apply this, we will typically encode graphs in configurations in various ways, and then look for cones that additionally satisfy some conditions (that are sufficiently quick to check). Often, the colors will actually be read in multiple nodes of another group (usually, the group $G$ on which configurations are laid).

In particular, in the case of complexity-theoretic results, by carefully choosing whether or not to include a node in the encoded graph, and choosing the meaning of ``successfulness'' so that it means that some type of computation is propagated, we can ensure that successful cones also satisfy some constraints, which will allow us to encode Turing machine computations on them.

\section{Nonembedding results}

We start with the nonembedding results, as they involve a simpler version of the construction. 

The following lemma is the basis of our nonembedding results. It is a simple way to express that the size of quotients needed in the definition of residual finiteness of $\Aut(A^{\Z^d})$ has an upper bound in terms of the radii of automorphisms.

\begin{lemma}
\label{lem:ResiduallyFinite}
Let $A$ be a finite alphabet, let $\mathcal{G} = \langle S \rangle \leq \Aut(A^{\Z^d})$, where $S$ is a finite generating set defining the word norm of $\mathcal{G}$. Then there exists $C \in \N$ such that whenever $\mathcal{H} \leq \mathcal{G}$ is a simple group containing a nontrivial element of norm at most $n$, we have $H \leq \Sym(C^{n^d})$. 
\end{lemma}

In particular, $H \leq \Sym(C^{n^d})$ implies that $H$ has cardinality at most $|C|^{n^d}!$.

\begin{proof}
Let $r$ be the maximal radius among generators in $S$. Suppose $h \in H$ has norm $n$. Then the radius of $h$ is at most $rn$. Since $h$ is nontrivial, the local rule of $h$ acts nontrivially on some pattern $p \in A^{B_{rn}}$. Thus, $h$ acts nontrivially on the set of points with period $(2r+1)n$ in each direction (since $(2r+1)n \geq 2rn + 1$ and $B_{rn} \subset [-rn, rn]^d$).

Thus, $h \in H$ has nontrivial image in the natural quotient $H \to \Sym(A^{((2r+1)n)^d})$ arising from the action on periodic points of such period. Since $H$ is simple, and the kernel of this map is not $H$, the kernel must be trivial, so $H$ embeds in $\Sym(|A|^{((2r + 1)n)^d}) = \Sym(C^{n^d})$ where $C = |A|^{(2r + 1)^d}$.
\end{proof}

Note that if $H$ has at least three elements, then $H$, being simple, must embed in the simple group $\Alt(C^{n^d})$.

We are mainly interested in the cardinality bound obtained for the simple group $\mathcal{H}$. For easier comparison of such bounds, we introduce the following definition.

\begin{definition}
The \emph{simplistic growth} of a finitely-generated group $\mathcal{G}$ is $f : \N \to \N$ where $f(n)$ is the maximal cardinality of a simple group $\mathcal{H} \leq \mathcal{G}$ containing a nontrivial element of norm $n$. 
\end{definition}

We sometimes say that a not necessarily finitely-generated group $\mathcal{G}$ has simplistic growth at most $f$. This means that every finitely-generated subgroup of $\mathcal{G}$ has simplistic growth at most $f$.

As with most growth-type functions, this definition depends on the generating set of $\mathcal{G}$, but the growth type does not. The following is immediate from Lemma~\ref{lem:ResiduallyFinite}. We also note that we could just as well use alternating groups in place of simple groups in this definition.

\begin{lemma}
\label{lem:UpperBound}
For any alphabet $A$, $\Aut(A^{\Z^d})$ has simplistic growth at most $e^{n^d}!$.
\end{lemma}

We now prove an almost-matching lower bound.


\begin{lemma}
\label{lem:LowerBoundTechnical}
Let $A$ be a nontrivial finite alphabet. Then $\Aut(A^{\Z^d})$ has a finitely-generated subgroup such that for large enough $n$, there is a nontrivial element of word norm $O(n \poly(\log(n))$ which belongs to an alternating group on at least $e^{\Omega(n^d)}$ elements.
\end{lemma}

\begin{proof}
We show first that if $D = d + 1$ and $A = C \times B$ where $B = \{0,1\}^d$ and $|C| \geq 6$, then the finitely-generated group $\PAut(A) \leq \Aut(A^{\Z^D})$ satisfies that for large enough $n$, there is a nontrivial element of word norm $O(n \poly(\log(n)))$ which belongs to an alternating group on $|C|^{\Omega(n^d)}$ elements. Identify $\Z^d$ with $\Z^d \times \{0\} \leq \Z^D$ in the obvious way.

For the following discussion, fix $n = 2^k \in \N$ and a configuration $x \in A^{\Z^D}$. The \emph{index} of $\vec v \in \Z^D$, written $\counters_x(\vec v)$ is the tuple of $d$ many numbers whose binary representations of length $k+1$ are written on the $B$-component (each cell contains one bit of each number), in the consecutive positions $(\vec v, \vec v + e_D, \ldots, \vec v + ke_D)$ of the configuration $x$. The components in this vector must all have a value in $[0, n] = [0, 2^k]$, or otherwise the index is not defined, i.e.\ we always have $\counters_x(\vec v) \in [0, n]^{d}$ if $\counters_x(\vec v)$ is defined. A coordinate $\vec v$ is \emph{successful} if its index $\counters_x(\vec v) \in [0, n]^d$ exists, and for each generator $e_i$ of $\Z^d$, the index of $\vec v + e_i$ exists and $\counters_x(\vec v + e_i) = \vec u + e_i$.

Let $S = \{1, \ldots, d\}$. 
To each configuration $x$, we can associate a good $S$-labeled $(C \times B)$-colored graph whose nodes are those nodes of $\Z^D$ where the index is defined (in the fixed configuration $x$), and for $1 \leq i \leq d$ we have an edge from $\vec v$ to $\vec v + e_i$ with label $i$ if both $\vec v$ and $\vec v + e_i$ are nodes of the graph (for goodness, we must formally have an inverse $i^{-1}$-edge backward, but we do not use it).

The successor relation of $i$ on $B$ is to increment the $i$th coordinate (we increment if the value is below $n$, and map to $\bot$ otherwise), which gives a partial action of $S^*$ on $B$. We have a corresponding rank function $\rho : B \to \N$ by simply summing the coordinates.

\textbf{Ripple catching.} The elements $\psi_{\pi, \ell, c}$ and $\gamma_{\pi, \ell}$ of the group $G_{[1, n-1], C, B}$ from Section~\ref{sec:RippleCatching} act in an obvious way on configurations of $A^{\Z^D}$: If $c$ is not in the support of $\pi \in \Alt(C)$, then at a node $\vec v$, $\psi_{\pi, \ell, c}$ will permute the $C$-component of $x_{\vec v}$ by $\pi$ if
\begin{itemize}
\item the index exists at $\vec v$,
\item $\rho(\vec) = \sum \counters_x(\vec v) = \ell$,
\item the node $\vec v$ is successful, and 
\item the $C$-value at $\vec v + e_i$ is $c$ for $i \in [1, d-1]$.
\end{itemize}
Note that the node $\vec v$ being successful means precisely that the indices of neighbors $\vec v + e_i$ exist, and are obtained from the index of $\vec v$ by incrementing the corresponding coordinates.

The generator $\gamma_{\pi, \ell}$ in turn applies $\pi$ if the index exists at $\vec v$, and $\rho(\counters_x(\vec v)) = \ell$. This gives indeed an action of $G_{S, C, B}$, because $G_{S, C, B}$ is defined by its action on any good graphs, and we are acting here by the same action on particular graphs (interpreted in configurations).


Lemma~\ref{lem:RippleCatching} tells us that there is an element $\psi$ in $\mathcal{G}_{S, C, B}$ whose word $n$-norm (with respect to the generators $\psi_{\pi, \ell, c}$ and $\gamma_{\pi, \ell}$ with $\ell \leq n$) is linear in $n$. The behavior of this element can be described as follows: The shape of the graph, or the $B$-component of a node, are never altered (because the generators never alter them). In the $C$-component of a node $u$, we apply $\pi$ if and only if $G_{\leq n}(u) = uS^{\leq n}$ (the cone of depth $n$ is full) and for all $u \neq v \in G_{\leq n}(u)$, the $C$-value is $0$.

In our specific situation, this translates to the following: Define the \emph{geometric cone} $T_n$ by 
\[ \vec u \in T_n \iff \vec u \in \N^d \wedge \sum \vec u \leq n. \]
We refer also to translates of $T_n$ as geometric cones. Then $\pi$ is applied to the $C$-component of the color at $\vec v$ if and only if the geometric cone $\vec v + T_n$ is \emph{self-indexing} meaning $\counters(\vec v) = 0^d$, and for any $\vec u \in T_n$, $\counters(\vec v + \vec u) = \vec u$; furthermore the $C$-value of $\vec v + \vec u$ must be $0$ whenever $\vec u \in T_n \setminus \{0^d\}$.

It now suffices to show two facts: $\psi$ belongs to an alternating group $\mathcal{H} \leq \Aut(A^{\Z^D})$ which acts faithfully on $|C|^{\Omega(n^d)}$ elements; and the lengths of $\psi_{\pi, \ell, c}$ and $\gamma_{\pi, \ell}$ are polynomial in $k$ in the group $\PAut(A)$ when $0 \leq \ell \leq n$. Namely, then the linear word norm of $\psi$ amounts to word norm $n \poly(k)$ in $\PAut(A)$, giving the claim.

\textbf{The simple group.} We start by describing $\mathcal{H}$: we observe that two self-indexing geometric cones $\vec v + T_i$ and $\vec u + T_i$ can only intersect for $\vec v = \vec u$, since the counter values determine the root (element with minimal coordinate sum) of the geometric cone uniquely. Thus, there is a well-defined action of the alternating group on the set $C^{T_n}$, which acts by even permutations on the configurations overlaid on self-indexing geometric cones. Clearly, $\psi$ above is an element of this group, and we have $|C^{T_n}| = |C|^{\Omega(n^d)}$.

\textbf{Small word norm from Barrington's theorem.} All that remains is to show that the lengths of $\psi_{\pi, \ell, c}$ and $\gamma_{\pi, \ell}$ are polynomial in $k$ in the group $\PAut(A)$. By the definition of $G_{S, C, B}$, it suffices to show that the generators $\phi_{\pi, c}, \beta_\pi, \gamma_{\pi,\ell}$ have such length.

For $\phi_{\pi, c}$, note that we could simply change our finitely-generated group and add it to $\PAut(A)$. However, as shown in Lemma~\ref{lem:PhiInPAut}, it is already in $\PAut(A)$ (with, of course, word norm independent of $n$).

For $\gamma_{\pi, \ell}$ and $\beta_\pi$ we proceed by finding an action of $\mathcal{G}_{md,C,\{0,1\}}$ on $A^G$ by elements of $\PAut(A)$ for suitable $m$, and then apply appropriate lemmas from the appendix. For the choice of $m$, we note that $\gamma_{\pi, \ell}$ and $\beta_\pi$ apply the permutation $\pi \in \Alt(C)$ on the $C$-value at $\vec u$ depending on some condition on the $B$-values of nodes $\vec u + \vec v$ where
\[ \vec v \in V = \{e_i + je_D \;|\; 1 \leq i \leq d, 0 \leq j \leq k \}. \]
We set $m = d (k+1)$. Each $1 \leq i \leq md$ corresponds bijectively to a tuple $\xi(i) \in V \times [1, d]$, where the $V$-component identifies the offset vector $\vec v$ and the last component is the coordinate of $B$ where the bit is read. 

Consider a generator $\pi_{b',i}$ of $\mathcal{G}_{md,C,B}$ with $\xi(i) = (\vec v, j)$. It acts as follows: if $x_{\vec u + \vec v}$ has $B$-value whose $j$th component is equal to $b'$, then apply $\pi$ to the $C$-value of $x_{\vec u}$. This (i.e.\ the automorphism by which it acts) is in $\PAut$: consider first the symbol permutation $\eta_{\pi'} \in \PAut(A)$, which $\pi' \in \Sym(C \times B)$ applies $\pi$ in the $C$-component if and only if the $j$th bit in the $B$-component has value $b'$. Then $\eta_{\pi'}$ and $\pi_{b',i}$ are conjugate by the partial shift by $\vec v$. Thus, $\pi_{b',i} \in \PAut(A)$ as claimed, and its word norm is clearly linear in $k$.

By Lemma~\ref{lem:Sum}, the permutation $\pi|\{\bin_m(n_1) \bin_m(n_2) \cdots \bin_m(n_d) \;|\; \sum_i n_i = \ell \}$ has word norm polynomial in $m = O(k)$, in the group $\mathcal{G}_{md,C,\{0,1\}}$. Up to ordering the coordinates (which trivially preserves word norm) this permutation acts precisely by the automorphism $\beta_{\pi, c}$. Since as elements of $\PAut(A)$ the $\pi_{b',i}$ have word norm linear in $k$, we conclude that $\beta_{\pi, c}$ has word norm polynomial in $k$.

For $\gamma_\pi$, the argument is similar, but using Lemma~\ref{lem:Increment}.

\textbf{General alphabet and no extra dimension.} We will now explain the changes needed to get the result for general alphabet $\Sigma$ and dimension $d$. To reduce the dimension to $d$, the idea is to lay down the coordinates in dimension $e_1$ rather than in a $(d+1)$th dimension. For the alphabet, the idea is to encode the larger alphabet $A$ in words over the (possibly) smaller alphabet. In both cases, we run into the issue that some configurations may not contain good encodings of a configuration. 

We may assume $0, 1 \in \Sigma$. Let $m$ be large, and let $W \subset \{0,1\}^m$ be a set of words of the form $110w0$ such that $w$ does not contain $11$ as a subword. Note that this set of words is mutually unbordered. For $m$ large enough, we can have a bijection $W \cong C \times B \times C$, $|C| \geq 6$, $B = \{0,1\}^d$.

Our generating set is an analog of the generators of $\PAut$, but using ``rotations of conveyor belts'' interpreted in maximal runs of the form $W^\ell$ in place of partial shifts. More precisely, let $x \in \Sigma^{\Z^d}$. Then for each $\vec v \in \Z^d$, the row $\vec v + \langle e_1 \rangle$ splits uniquely into nonoverlapping occurrences of words $W$. Whenever a word of $W$ appears in $\{\vec v, \vec v + e_1, \ldots, \vec v+m-1\}$, $\vec v$ is a \emph{base} in $x$. Then $\vec v + je_1$ is not a base for any $0 < |j| \leq m$.

Define $\bar e_1 = me_1$, $\bar e_j = e_j$ for $j \geq 2$. If $\vec v$ and $\vec v + \bar e_i$ are bases, we interpret a forward edge with label $i$. Assign backward edges accordingly. This gives $x$ the structure of a graph. Each base corresponds to an occurrence of a word from $W$, thus a tuple $(c, b, c') \in C \times B \times C$.

Seeing $(c, b, c')$ as a column vector $\begin{pmatrix} c \\ b \\ c'\end{pmatrix}$, there is a natural conveyor belt action in each direction: Starting from the top element $c$, follow forward and backward edges labeled $i$ from the base $\vec v$, in a dense set of configurations we obtain a finite word
\[ \begin{pmatrix} c_0 & c_1 & c_2 & \cdots & c_{k-1} & c_k \\ b_0 & b_1 & b_2 & \cdots & b_{k-1} & b_k \\ c_0' & c_1' & c_2' & \cdots & c_{k-1}' & c_k' \end{pmatrix} \]
(again, these are interpreted in the words $W$ at the bases). Rewrite this to (the words in $W$ corresponding to symbols)
\[ \begin{pmatrix} c_0' & c_0 & c_1 & \cdots & c_{k-2} & c_{k-1} \\ b_0 & b_1 & b_2 & \cdots & b_{k-1} & b_k \\ c_1' & c_2' & c_3' & \cdots & c_k' & c_k \end{pmatrix} \]
This gives a shift-commuting uniformly continuous function on the dense set of configurations where these words are always finite, which has a two-sided inverse (the corresponding backward rotation), and thus it extends to a shift-commuting continuous function $\sigma_{i,C}$ on all of $\Sigma^{\Z^d}$, with a two-sided inverse. 

As generators of our group, we can take the $\{\sigma_{i,C} \;|\; 1 \leq i \leq d\}$ and all symbol permutations $\eta_\pi$.

Now let us explain how to modify the previous proof. We will never permute the $C$s on the bottom track. If we have a base at each of $\vec v, \vec v + \bar e_1, \ldots, \vec v + k \bar e_1$, then we can interpret a $d$-dimensional vector by reading the bits on the $B$-component. This vector forms the \emph{index} of $\vec v$. Write $\hat e_1 = k\bar e_i (kme_i)$ and $\hat e_i = \bar e_i (= e_i)$ for $i \geq 2$.

We now want to perform controlled permutations on the $C$'s on the top track. We concentrate on a single $C$-symbol on the top track, which we think of as moving around on conveyor belts as we apply powers of $\sigma_{i, C}$. We have in our generating set a controlled permutation that permutes $C$ (by any even permutation) if and only if it is on the top track. Conjugating such a controlled permutation by partial shifts, we can check that a certain power of $\sigma_{i, C}$ keeps it on the top track.

We can then use the techniques from Appendix~\ref{sec:Barrington}, to control the permutation by polynomially many such checks, by a divide-and-conquer commutator construction. This gives us an element of word norm $\poly(k)$ which permute a particular $C$-symbol nontrivially if and only if it is on the top track of its conveyor belt, and the belt continues by a  sufficient length in the forward direction. In particular from the case $i = 1$, we obtain that we can condition a permutation on the coordinate having an index.

We can also check that the $i$th edge exists in the graph described above, between two bases, by checking (by commutatoring with yet another controlled permutation) that after applying $\sigma_i$, our $C$-symbol is still in a base. 

We can then perform a permutation analogous to $\phi_{\pi, c}$, which applies a permutation $\pi$ on the $C$-symbol at base $\vec v$ if and only if the $i$-edge exists, and the $C$-symbol at $\vec v + \hat e_i$ is $c$. Namely, permutations of the $B$-symbol conditioned on the $C$-symbol (or indeed any of our generators) do not change the structure of the graph, so we can use the scheme of Lemma~\ref{lem:PhisInPAut}. The main subtlety is that it may happen that the $\hat e_i$-rotation brings $C$ back to its starting position, for example, but then we will simply then permute it by $\pi$ controlled by its own value (and thus will not modify its value). This is because ultimately the logic of the lemma is in terms of the set of positions that see a $c$-symbol when the partial shift is applied.

We can also perform similar checks on the indices as in the previous scheme. Note that if the graph is not of the correct form, the ``intermediate calculations'' (controlled permutations from which the final one is composed of) may be totally meaningless. For example, it might be that $\vec v$ has an index, but $\vec v+\bar e_2$ does not even have a full conveyor belt. However, these meaningless controlled permutations will ultimately cancel in the commutator formulas unless all of the checks succeed (in which case the structure must have been meaningful, so all the value checks are meaningful as well).

Now define a new graph structure, where we take as nodes only those bases where the index exists, and has value in $[0, 2^k]^d$,\footnote{Note that indices can overlap, in that both $\vec v$ and $\vec v + \bar e_1$ may well have a well-defined index. We could use an additional marker symbol to ensure that this does not happen. However, we are only checking some property of the configuration of indices, and not changing their values (except temporarily for the implementation of the $\phi_{\pi, c}$-analog above), so this does not matter, and although it is perhaps not intuitive to allow them to overlap, it is formally simpler.} and the induced edges. In this new graph, we have natural $C \times [0, 2^k]^d$-colors for the nodes, taking the $C$-color at $\vec v$ (in a single cell) and the index at $\vec v$ (in the $W$-words starting at $\bar e_1, \ldots, k \bar e_\ell$).

As explained above, we now have all the tools to implement scheme from above, using $\sigma_{i, C}$ and symbol permutations.
\end{proof}

\begin{lemma}
\label{lem:LowerBound}
Let $A$ be a nontrivial finite alphabet. Then $\Aut(A^{\Z^d})$ has a finitely-generated subgroup with simplistic growth at least $e^{n^d/\poly\log(n)}$.
\end{lemma}

\begin{proof}
Let $f$ be the simplistic growth function. By the previous lemma we have
$f(Cn \log^t(n)) \geq e^{\Omega(n^d)}$ for some fixed $C$ and all $n$, so by Lemma~\ref{lem:Invert} we have $f(n) \gtrsim e^{n^d/\poly\log(n)}$ ($\gtrsim$ swallows $\Omega$, and $\poly$ swallows $d$).
\end{proof}

\begin{theorem}
\label{thm:HochmanSolution}
Suppose that $D > d$, and let $A,B$ be nontrivial finite alphabets. Then $\Aut(A^{\Z^D})$ has a finitely-generated subgroup that does not embed in $\Aut(B^{\Z^d})$.
\end{theorem}

\begin{proof}
By Lemma~\ref{lem:LowerBound}, $\Aut(A^{\Z^D})$ has a finitely-generated subgroup with simplistic growth at most $e^{n^D/\poly\log(n)}!$. If $\Aut(A^{\Z^D}) \leq \Aut(B^{\Z^d})$, then this finitely-generated subgroup would have simplistic growth at least $e^{n^d}!$ by Lemma~\ref{lem:UpperBound}, a contradiction since
\[ e^{n^D/\poly\log(n)}! \not\lesssim e^{n^d}!. \] 
\end{proof}


\subsection{Dropping the centers}

By completely analogous arguments, one can prove variants of the results where the center is quotiented out. We omit a technical version of Lemma~\ref{lem:ResiduallyFinite}, and explain the modifications needed directly.

The following is due to Hochman \cite{Ho10} in general dimension. The one-dimensional result is due to Ryan \cite{Ry72}.

\begin{lemma}
\label{lem:Center}
The center of $\Aut(A^{\Z^d})$ is the group of shifts.
\end{lemma}

\begin{theorem}
\label{thm:Center}
Suppose that $D > d$, $\mathcal{G}_d = \Aut(A^{\Z^d}), \mathcal{G}_D = \Aut(B^{\Z^D})$ for any nontrivial alphabets $A, B$. Then $\mathcal{G}_D$ has a finitely-generated subgroup such that $\mathcal{G}_D/K$ does not embed in $\mathcal{G}_d/K'$ for any central subgroups $K \leq \mathcal{G}_D, K' \leq \mathcal{G}_d$.
\end{theorem}

\begin{proof}
The proof is analogous to Theorem~\ref{thm:HochmanSolution}. Let $\mathcal{G} = \langle S \rangle$ be the group used in that proof. Then in $\mathcal{G}_D/K$ we take the the subgroup $\mathcal{G}/K$.

A we showed, there is a subgroup $\mathcal{H}$ of $\mathcal{G}_D$ which is nonabelian simple of order $e^{\Omega(n^D)}!$, and contains an element of word norm $n' = O(n \poly(\log(n))$. Now, $\mathcal{H}$ must have trivial intersection with $K$,\footnote{Here note that since $\Z^d$ is abelian, shift maps are automorphisms, and we identify them with a subgroup of $\Aut$.} since otherwise $K \cap \mathcal{H}$ is a nontrivial simple subgroup of $\mathcal{H}$. Thus by the second isomorphism theorem, $\mathcal{H}K/K \cong \mathcal{H}$ as a subgroup of $\mathcal{G}/K$, and we conclude that $\mathcal{G}/K$ still has a simple subgroup of order $e^{\Omega(n^D)}!$ and with an element of word norm $n'$.

Next we observe that the set of groups embeddable in some $\mathcal{G}_d/K'$ (over all alphabets $B$) does not decrease if we assume $K'$ is just the subgroup $\langle e_1, \ldots, e_{d'} \rangle$ of $\Z^d$ for some $d' \leq d$. Namely, we observe first (e.g.\ by linear algebra) that there is a subgroup $L$ such that $G = K' \oplus L$ is of finite index in $\Z^d$. Any subgroup of $\mathcal{G}_d/K' = \Aut(B^{\Z^d}, \Z^d)/K'$ is certainly a subgroup of $\Aut(B^{\Z^d}, G)/K'$. Then we recall that the $G$-action $(B^{\Z^d}, G)$ itself is topologically conjugate to a full shift on symbols $B^{[\Z^d : G]}$. By replacing the alphabet $B$ by this set and observing that $G \cong \Z^d$ (by the fundamental theorem of abelian groups), we may replace the original group $\Z^d$ by $G$. Since $K' \times L$ is now an internal direct product in $\Z^d$, by the fundamental theorem of abelian groups $K' \cong \Z^{d'}$ and $K' \cong \Z^{d''}$ for some $d' + d'' = d$, and thus by changing the basis we may assume $K'$ is just the subgroup $\langle e_1, \ldots, e_{d'} \rangle$.

Assume now that $K'$ is of this form, and let $L = \langle e_{d'+1},\ldots, e_{d''} \rangle$. Consider a finitely-generated subgroup $H'$ of $\mathcal{G}_d/K'$. We may take representatives for the generators $\mathcal{S}' \subset \mathcal{G}_d$, and suppose $r$ is a bound on their radii. Suppose that some nontrivial element $h \in \mathcal{H}'$ of this group has word norm $n'$ and belongs to a nonabelian simple group of order $e^{\Omega(n^D)}!$. Let $h'$ be the corresponding representative in $\mathcal{G}_d$ (of the same word norm). Since $h$ is nontrivial, $h'$ is not contained in $K'$. Note that the radius of $h'$ is at most $rn'$.

We observe that $\mathcal{G}_d/K'$ has a well-defined action on the set of $K'$ orbits of points, and thus so does $\mathcal{H}'$. We claim that we can find a point with periods at most $10rn'$ in each direction, such that $h'$ does not act trivially on its $K'$-orbit. Since $\mathcal{H}'$ is simple, as in the proof of Lemma~\ref{lem:ResiduallyFinite} we see that $H'$ must act faithfully on the set of $K'$-orbits with periods at most $10rn'$, and we conclude that again $|\mathcal{H}'| \leq e^{O((10rn')^d)}! = e^{O((n')^d)}!$ (the invisible constant is only a function of $r$ and the alphabet $B$). This concludes the proof by again observing that $e^{O((n')^d)}!$ is far smaller than $e^{\Omega(n^D)}!$.

We just need to prove the claim about the action on periodic points. Suppose $B \ni 0, 1$. There are two cases to consider. First, suppose $h'$ is a shift by an element not in $K'$, then it is a shift by a vector of radius at most $rn$ from outside $K'$. We can easily find a point with periods $(2rn'+1)$ where $h'$ does not act by a $K'$-shift, for example we can pick the point $x \in \{0,1\}^{\Z^d}$ with periods $2(rn'+1)\Z^d$, where in the fundamental domain $[0, 2rn']^d$ we have $x_{\vec v} = 1 \iff \vec v = \vec 0$, and then $x$ is not mapped into its $K'$-shift orbit.

Suppose then that $h'$ is not a shift map. Let first $y \in \{0,1\}^{\Z^d}$ be the characteristic function of $\{\vec 0\}$, i.e.\ $y_{\vec v} = 1 \iff \vec v = \vec 0$. Then we have two subcases. First, if $h'(y) \notin K' \cdot y$, then we can pick the same point $x \in \{0,1\}^{\Z^d}$ as above.

Suppose then $h'(y) \in K' \cdot y$, so $h'(y) = \sigma_{\vec v}(y)$ for some $\vec v \in B_{rn'} \cap K'$. Since $h'$ is not a shift by an element of $K'$, we can find a pattern $P \in B^{B_{rn'}}$ such that $h'(z) \neq \sigma_{\vec v}(z)$ when $z \in [P]$. Positioning $z|B_{rn'}$ and a central pattern of $y$ far from each other in the same $10rn\Z^d$-periodic point $x$, we obtain a point which $h'$ does not shift by an element of $K'$. 
\end{proof}

\subsection{Embeddings into other groups}

A version of Lemma~\ref{lem:ResiduallyFinite} can be proved more generally for residually finite groups.

We recall the definition of systolic growth of a group \cite{BoCo16}.

\begin{definition}
Let $G$ be a group. Its \emph{systolic growth} with respect to a generating set $S$ is the function $f : \N \to \N$, where $f(n)$ is the minimal index of a subgroup $H$ such that the cosets $Hg$ are different for distinct $g \in B_n$.
\end{definition}

Usually, one defines the systolic growth instead by taking $f'(n)$ the minimal index of a subgroup $H$ such that $H \cap B_n = \{1_G\}$. However, it is easy to check that 
\[ f'(n) \leq f(n) \leq f'(2n) \]
so the functions are asymptotically equal.

\begin{definition}
Let $G$ be a group and $X \subset A^G$ a subshift. Its \emph{periodization growth} with respect to a generating set $S$ is the function $f : \N \to \N$ such that for any $p \in X|B_n$, there exists a point $x \in X$ with $x|B_n = p$ whose shift orbit has cardinality at most $f(n)$.
\end{definition}

\begin{lemma}
If $G$ has systolic growth $f$, then $A^G$ has periodization growth $f$.
\end{lemma}

\begin{proof}
Let $p \in A^{B_n}$ be arbitrary. Let $H \leq G$ be a subgroup with index $f(n)$ such that the cosets $Hg$ for $g \in B_n$ are distinct. Extend $B_n$ to a set of right representatives $R$ for cosets of $H$, and pick any $q \in A^R$ with $q|B_n = p$. Define $x \in A^G$ by $x_{hk} = q_k$ for $h \in H$, $k \in R$. Then clearly $x|B_n = p$ and $x$ is $H$-periodic because $hx_g = x_{h^{-1}g} = x_g$ for $h \in H$. Therefore, $|Gx| \leq [G : H] = f(n)$.
\end{proof}


\begin{lemma}
\label{lem:ResiduallyFiniteGeneral}
Let $G$ be a group, $X \subset A^G$ a subshift, and $\mathcal{G} = \langle S \rangle \leq \Aut(X)$, wheret $S$ is a finite generating set defining the word norm of $\mathcal{G}$. Suppose $X$ has periodization growth $f$. Then there exists $r \in \N$ such that whenever $\mathcal{H} \leq \mathcal{G}$ is a simple group containing a nontrivial element of norm at most $n$, we  have $\mathcal{H} \leq \Sym(|A|^{f(rn)})$.
\end{lemma}

\begin{proof}
Let $r$ be the maximal radius among generators in $S$. Suppose $h \in \mathcal{H}$ has norm $n$. Then the radius of $h$ is at most $rn$. Since $h$ is nontrivial, the local rule of $h$ maps some pattern $p \in X|B_{rn}$ to a symbol distinct from $p_e$.

Let $x \in X$ be a point with $x|B_n = p$ whose shift orbit has cardinality at most $f(n)$, and let $H$ be the stabilizer (of index at most $f(n)$) of $x$. Let 
\[ P = \{y \in X \;|\; \forall h \in H: hy = y\}. \]
Since the natural action of $\Aut(X)$ commutes with the shift action, in particular $\mathcal{H}$ stabilizes $P$ (as a set). We have $hx \neq x$ since $x|B_n = p$ and by the assumption on the local rule of $h$. Since $H$ is simple, its action on $P$ is faithful. We conclude that $H \leq \Sym(P)$, where of course $|P| \leq |A|^{f(rn)}$.
\end{proof}

Of course, this can be improved by taking into account the growth of the number of patterns in $X$ that appear in $H$-periodic points, where $H$ ranges over the group used to realize periodization growth.

We again state the previous lemma in terms of simplistic growth.

\begin{lemma}
\label{lem:GeneralUpperBound}
Let $G$ be a group, and let $X$ have periodization growth $f$. Then every finitely-generated subgroup of $\Aut(X)$ has simplistic growth at most $e^{f(n)}!$.
\end{lemma}

We show as an application that the group of $5$-dimensional cellular automata does not embed in the cellular automata group on the Heisenberg group.

\begin{proposition}
\label{prop:Heisenberg}
Let $H_3$ be the Heisenberg group of upper triangular $3$-by-$3$ integer matrices with unit diagonal. then $\Aut(A^{\Z^D})$ does not embed in $\Aut(B^{H_3})$ for $D \geq 5$.
\end{proposition}

\begin{proof}
The group $H_3$ has $n^4$ systolic growth \cite{BoCo16}. Thus, $B^{H_3}$ has periodization growth $n^4$ by Lemma~\ref{lem:PeriodizationGrowth}. By Lemma~\ref{lem:LowerBound}, $\Aut(A^{\Z^D})$ has a finitely-generated subgroup with simplistic growth at least $e^{n^D/\poly\log(n)}$ with $D \geq 5$. If $\Aut(A^{\Z^D}) \leq \Aut(B^{H_3})$, then this finitely-generated subgroup would have simplistic growth at most $e^{n^4}!$ by Lemma~\ref{lem:GeneralUpperBound}, a contradiction since
\[ e^{n^D/\poly\log(n)}! \not\lesssim e^{n^4}!. \]
\end{proof}

More generally, every nilpotent group has polynomial systolic growth. In fact, in \cite{BoMc11} it is shown that even the residual girth is polynomial, and this function bounds the systolic growth from above.

\section{Groups with free submonoids}
\label{sec:FreeSubmonoids}

Groups with free submonoids admit explicit cones: binary trees along two generators. Furthermore, since the free submonoid on two generators contains the one on three generators, we have a free direction where we can store any additional information. We show that this implies a lower bound on simplistic growth. We also show that by a simple modification of the notion of successfulness, the same proof gives PSPACE-hardness of the word problem of the group of cellular automata. Both of these results are generalized later, after we interpret graphs in configurations more abstractly in Section~\ref{sec:Arboreous}.

\subsection{Simplistic growth lower bound}

\begin{theorem}
\label{thm:SimplisticLowerBoundGeneral}
Let $G$ be a f.g.\ group containing a free submonoid on two generators, and let $A$ be any alphabet. Then $\Aut(A^G)$ has simplistic growth at least $e^{2^{\sqrt[\ell]{n}}}!$ for some $\ell$. Thus it has an f.g.\ subgroup that does not embed into $\Aut(B^{\Z^d})$ for any alphabet $B$ and $d \in \N$. 
\end{theorem}

\begin{proof}
We consider the case $A = C \times \{a, b, ?\}$ with $|C| \geq 6$. In the case of a general alphabet, the marker/conveyor belt argument is similar as in Lemma~\ref{lem:LowerBoundTechnical}. Note also that in the interesting special case of a free group $G$, $\Aut(A^G)$ actually embeds in $\Aut(B^G)$ for nontrivial alphabets $A, B$ \cite{BaCaRi25}.

We give a lower bound on simplistic growth far surpassing that of any $\Aut(B^{\Z^d})$. If there is a free monoid on two generators, there is also a free submonoid on three generators $a, b, c$. Given a configuration $x \in A^G$, we construct a graph for ripple catching. At $g \in G$ read the word $\{a, b, ?\}^n$ at $\{g, gc, gc^2, \ldots, gc^{n-1}\}$. If the word is of the form $u?^k$ where $u \in \{a, b\}^{n-k}$. Then we take $g$ as a node with color $u$. Note that this color is an analog of the index from the $\Z^d$ case, and allows a unique color for each position in a binary tree.

We have a $t$-labeled edge from $g$ to $gt$ if both nodes exist, for $t \in \{a, b\}$. We have a natural partial action of $\{a, b\}^*$ which appends the generator at the end of the vertex color, and we use the actions of $a, b$ as their successor relations. The corresponding rank function is $\rho(w) = |w|$. Note that a node $g \in G$ being successful means that its color is of the form $w$ for $w \in \{a, b\}^k$, and for each $t \in \{a, b\}$, the color of $gt$ is precisely $w?^{k-1}$ (nodes $k = 0$ are not successful).

Again, $\PAut$ implements the generators of the ripple group for this graph (with respect to the rank function and partial action of $\{a, b\}^*$), so that the generators of the ripple group have word norm $\poly(n)$. The only difference is that instead of incrementing numbers (or more precisely, checking that some number is the successor of another), we now append a symbol at the end of a word. The latter is much easier to implement (indeed, it is a matter of comparing individual symbols).

Then by Lemma~\ref{lem:RippleCatching}, there is an element of word norm $\poly(\poly(n)) = \poly(n)$ which acts nontrivially exactly on full cones, which contain $0$ everywhere except at the root. Again, cones are self-indexing in the sense that the vertex color uniquely identifies a root of the cone, so this is an element of the simple group of all permutations of $C$-configurations on self-indexing cones.

Now the cones have $2^{n+1}-1$ nodes, so the cardinality of this simple group is
$|C|^{2^{n+1}}!/2$, so the simplistic growth $f$ of $\Aut(A^G)$ satisfies $f(n) \gtrsim e^{2^{\sqrt[\ell]{n}}}!$, far surpassing the simplistic growth $e^{n^d}!$ in the $d$-dimensional case.
\end{proof}

\subsection{PSPACE-hardness of the word problem}

\begin{theorem}
\label{thm:MonoidCase}
Let $G$ be a group containing a free monoid on two generators. Then for some $A$, the group $\PAut(A)$ has PSPACE-hard word problem.
\end{theorem}

\begin{proof}
We work exactly as above, but now in addition to knowing their coordinates in the cone, we include a configuration for a pair of Turing machines (the PSPACE-complete problem we reduce involves two Turing machines), and we include in the notion of successfulness of a node (in other words, in the definition of the successor relations of $a, b$ on the color) that the Turing machines are performing computations. (Note that the Turing machines are nondeterministic, so the successor relations do \emph{not} come from a partial action of $\{a, b\}^*$ in this case.)

Ripple catching is give us an element that acts nontrivially exactly on the roots of full computation trees. Such an element is nontrivial if a full computation tree exists, which corresponds exactly to a positive instance of the problem we are reducing.

We explain this in more detail. If $G$ contains a free monoid on two generators, it contains a free monoid on three generators, say powers of $a, b, c \in G$ satisfy no nontrivial positive relations.

We will use alphabet $A = C \times B \times (S \cup (Q \times S))$ for our f.g.\ group, where $B = \{a, b, ?\}$. Here, $Q$ with $|Q| = t$, and $S$ with $|S| = s$ are respectively the common sets of states and tape symbols of the nondeterministic Turing machines $T_a, T_b$ from Lemma~\ref{lem:PSPACEByTrees}. The lemma states roughly that the language of words $u_{\init} \in S^n$ such that there exists a binary tree of depth $n$ of configurations of the machines, where in one direction we take steps by $T_a$, and in the other by $T_b$, the root of the tree contains the initial configuration corresponding to $u_{\init}$, and both machines accept at all leaves.

Our reduction takes a word $u_{\init}$ of length $n = 2^k$ (note that one should think of this as a polynomial size input whose size just happens to be of nice form, rather than an exponential length word), and constructs an automorphism that performs a permutation $\pi$ on the $C$-component of $A$ whenever it sees a tree of the form described in Lemma~\ref{lem:PSPACEByTrees} on the configuration, encoded in words read on the $B$-components, and with $q_{\init} u_{\mathrm{init}} \in Q S^n$ as the initial configuration on the tree (where $q_{\init}$ is the common initial state of $T_a, T_b$).

We construct this by ripple catching, so we need to describe a graph where full cones correspond to valid computations trees. At $g \in G$, read words $(u, w)$ in coordinates $(g, gc, \ldots, gc^{n-1})$ where $u \in B^n$, $w \in (S \cup (Q \times S))^n$ (ignoring the $C$-components). We include the node in the ripple catching graph if the following hold:
\begin{itemize}
\item $u$ is as in the previous proof, i.e.\ all the $?$-symbols are at the end,
\item $w$ is a valid Turing machine configuration (has exactly one head),
\item if $u = ?^n$, then $w$ is the initial configuration corresponding to $u_{\init}$,
\item if $u \in \{a, b\}^n$, then $w$ is a final configuration.
\end{itemize}
We include an edge from $g$ to $ga$ with label $a$ if both are nodes.

We have $t \in \{a, b\}$ act as in the previous proof on the word $u$ (i.e.\ by appending itself), and in the Turing machine component $w$, $t \in \{a, b\}$ acts by taking a step of the machine $T_t$.

We now implement successfulness in $\PAut$ exactly as before. The only new component is that we need to check that the Turing machine steps are performed correctly, to determine whether a node is successful. This is possible by Lemma~\ref{lem:TuringStep}. All in all, checking successfulness requires word length polynomial in $n$. Combined with ripple catching, we have an element of word norm $\poly(\poly(n)) = \poly(n)$, which acts nontrivially if and only if a full cone exists, in other words if and only if $u_{\init}$ is in the language $L$.
\end{proof}

\section{Arboreous graph families}
\label{sec:Arboreous}

In Section~\ref{sec:GAP}, we show that assuming the Gap Conjecture, nice enough SFTs on groups with superpolynomial growth have automorphism groups with PSPACE-hard word problem. The idea is to interpret trees (i.e.\ free monoids) in these SFTs. We cannot quite find full binary trees (this would be impossible in any group of subexponential growth), but from the growth rate assumption we can at least deduce the existence of polynomial-diameter trees that have a linear diameter binary tree as a path contraction (i.e.\ up to replacing paths with edges, we do have full binary trees).

In this section, we define the desired shapes of such trees in more detail, define an analog of $\PAut$ on graph families, and show that as soon as a graph family contains the right kind of trees, $\PAut$ groups can have PSPACE-hard word problems.

Recall the definition of an $S$-labeled graph from Section~\ref{sec:RippleCatching}. We need some definitions for $S$-labeled graphs. Such a graph is \emph{complete} if all $S^{\pm}$-successors exist. A \emph{weakly connected component} of an $S$-labeled graph is just a connected component of the underlying undirected graph (with an undirected edge for each directed edge), and an $S$-labeled graph is weakly connected if it has just one weakly connected component.

\begin{definition}
For $|S|\geq 2$, a \emph{depth-$k$ $S$-ary word tree} is an $S$-labeled graph obtained as follows: Start with the graph with vertices $S^{\leq k}$ and an $s$-labeled edge $(u, us)$ for $s \in S, u \in S^{\leq k-1}$. Then for some $U \subset S^{\leq k}$, identify the vertices $\{uav, ubv, ucv\}$ whenever $u \in U, v \in S^*$, i.e.\ we quotient the tree by the minimal equivalence relation making these identifications. We say such a tree is an \emph{$(n,S)$-thicket} if for each $u \in S^k$, there are at least $n$ prefixes $v \leq u$ such that $|\{vs \;|\; s \in S\}| = |S|$ in the resulting tree.
\end{definition}

We phrase the definition in this constructive way rather than descriptively, since otherwise it is difficult to state concisely that the tree ``should not have any identifications other than ones arising from $ua = ub = uc$'' (as certainly such an identification leads to many other identifications, but we only want to have the obvious consequences). One indeed obtains a full ternary tree from an $(n,\{a,b\})$-thicket by contracting the paths along vertices $u$ where $ua = ub$.

\begin{definition}
Suppose $S \supset \{a, b, c\}$. A family of $S$-labeled graphs $\mathcal{G}$ is \emph{arboreous} if there exist $C,d$ such that for all $n$, there exists $H \in \mathcal{G}$ and $h \in H$ such that the subgraph of $H$ induced by the vertices $h\{a, b, c\}^{\leq C n^d}$ is an $(n, \{a, b, c\})$-thicket.
\end{definition}



Next we define our $\PAut$ analog.

\begin{definition}
Let $\mathcal{G}$ be a family of complete $S$-labeled graphs, let $F_S$ be the free group generated by $F_S$. For finite sets $A_i$, the group $\PAut_{\mathcal{G}}(A_1, \ldots, A_k)$ is defined by its action on $\bigsqcup_{G \in \mathcal{G}} A^G$ where $A = A_1 \times \cdots \times A_k$. The generators are as follows:
\begin{itemize}
\item For each $i \in [1, k]$ and $g \in F_S$, we have the partial shift $\sigma_{g, i}$ that maps $x = (x_1, \ldots, x_k) \in A^G$ to $(x_1,\ldots, x_{i-1}, y, x_{i+1}, \ldots, x_k)$, where $y_h = (x_i)_{hg}$ for $h \in V(G)$.
\item For each $\pi \in \Alt(A)$, we have the symbol permutation $\eta_\pi$ that maps $x \in A^G$ to $y$ defined by $y_h = \pi(x_h)$ for $h \in V(G)$.
\end{itemize}
\end{definition}

We do not require that $\mathcal{G}$ has a unique copy of each graph up to isomorphism, but requiring this would not change the definition.

\begin{theorem}
\label{thm:AbstractPSPACEComplete}
Let $\mathcal{G}$ be any arboreous path family. Then for some $A$, the group $\PAut_\mathcal{G}(A)$ has PSPACE-hard word problem.
\end{theorem} 

\begin{proof}
The proof is analogous to that of Theorem~\ref{thm:MonoidCase}. We use $A = C \times B$ with $|C| \geq 6$ and $B = \{a, b, ?\}$.

First, consider the case where $G$ contains a graph where for some vertex $h$, $h\{a, b, c\}^*$ is a copy of the ternary free monoid. In this case, we can directly use the element constructed in the proof of Theorem~\ref{thm:MonoidCase}, and perform ripple catching along the encoded free monoids $\{a, b\}^*$, interpreting indices in initial segments of the cosets $\langle c \rangle$. For inputs $u$ where a witnessing tree exists (in the sense of Lemma~\ref{lem:PSPACEByTrees}, proving $u \in L$), such a tree exists in some graph, and thus the $\PAut$-element is nontrivial.

On the other hand, if a witnessing tree does not exist abstractly, one cannot be embedded in graphs of any shape. Thus, as long as the ripple group is correcly implemented, we will have a correct reduction. There is now a possible source of worry, namely we could certainly have $uc^i = vc^j$ for distinct $u, v \in \{a, b\}^*$ and $i, j$. Such reuse does not matter as long as we are only permuting $C$-symbols as a function of the $B$-track. For the implementation of $\phi_{\pi, c}$, the $B$-symbols need to be changed as a function of $C$. In this case, we are interested in permuting the $C$-component of the color at a node $h \in H$ (where $H \in \mathcal{G}$) if the color at some relative position $ht$ for $t \in \{a, b\}^*$ contains $c \in C$. The proof of Lemma~\ref{lem:PhiIsInPAut} goes through verbatim for implementing this by partial shifts and symbol permutations.


We conclude that the controlled permutation constructed in the end acts nontrivially precisely in the presence of a tree embedded on the graph. 

Second, consider the general case, i.e.\ that we only have thickets with possibly polynomial paths between branchings in the trees. We still perform ripple catching, but now along the graph where long non-branching paths are contracted to single edges, and we need to again show that the ripple group can be implemented in a finitely-generated subgroup.

The first thing we need is to have a way to check whether some node $h \in H$ is branching, or more generally whether $hs = hs'$ for $s, s' \in F_S$. We can construct with the usual commutator techniques first an automorphism that permutes the $C$-component at $h$ if the $B$-tracks at $hs$ and $hs'$ have distinct values. Then we change the value at $hs$ and try again, and finally change the value back. One of these permutations succeeds if and only if $hs$ and $hs'$ are distinct. The details are similar to Equation~\ref{eq:eq}. The usual commutator formulas can then be used to obtain an automorphism that permutes the $C$-component at $h$ by a particular even permutation if and only if $hs \neq hs'$. 


Note that the contracted paths are of length $n^d$ at most. We explain first how to deal with a single length $\ell$. Another variant of Equation~\ref{eq:eq} gives an automorphism $f$ that performs $\pi$ in the $C$-component at $h$ if and only if $ha^\ell$ contains the symbol $c$. Namely, we simply replace the conjugating partial shifts of the third track by their $\ell$th power.

However, we need to check that the intermediate nodes on the path, i.e.\ $ha^{\ell'}$ for $\ell' < \ell$, do not branch, and that $ha^{\ell}$ does branch. This can be done since as we showed above, we have in our group automorphisms that permute $C$ in an arbitrary (even) way depending on whether a given relative node is branching. Specifically, one can use commutator formulas and divide-and-conquer to obtain a polynomial-norm automorphism that cancels the effect of $f$ if the branching is not correct.

The resulting automorphisms can be simply composed for all distinct $\ell$ to obtain the desired analogs of the $\phi$-maps, as their supports are distinct (the $\ell$th map can only act nontrivially when the non-branching prefix of the tree starting from $h$ is of length exactly $\ell$).

To allow successfulness checks to jump over paths, the argument is intuitively direct from Barrington's theorem (taking the maps that check for branching nodes to be to be among the generators), in the sense that NC$^1$ is certainly strong enough a computational model to allow skipping such polynomial length paths. Alternatively, one can work analogously as in the construction of the $\phi$-automorphisms, and for each $\ell$ and $0 \leq j \leq n$ construct a polynomial word norm controlled permutation that checks a value in $h a^\ell c^j$, and with commutator formulas have it cancel out unless the non-branching path from $h$ has length exactly $\ell$, and finally compose this over all the polynomially many $\ell$.
\end{proof}

\begin{theorem}
\label{thm:AbstractSimplisticGrowthLowerBound}
Let $\mathcal{G}$ be any arboreous path family. Then for some $A$, the group $\PAut_\mathcal{G}(A)$ has simplistic growth at least $e^{2^{\sqrt[\ell]{n}}}!$ for some $\ell$.
\end{theorem} 

\begin{proof}
The proof is essentially a direct combination of the above proof (for communication over long paths), and Theorem~\ref{thm:SimplisticLowerBoundGeneral}.

The simple group we implement can be described as follows: For every $(n, \{a, b, c\})$-thicket encoded in the configuration, in the self-indexing form from the lemma, we act on the bits on the \emph{branching positions}. It is crucial here that the group does not act on all of the bits on the encoded thicket, as otherwise it could be a direct product of actions on thickets of different shapes (which can have different numbers of nodes), and not simple at all. The element of this group that we construct is the one that performs a particular even permutation at the root of such a thicket, assuming all the branching positions contain the bit $0$.

This is immediately obtain from the above proof, by simply dropping the Turing machines and only checking the indices.
\end{proof}

\section{Arboreous groups}
\label{sec:GAP}

A sequence $b_k$ is at most \emph{polynomial} if $b_k = O(k^d)$ for some $d$. A function $b_k$ is at least \emph{stretched exponential} if we have $b_k = \Theta(e^{k^\beta})$ for some $\beta > 0$, and $\beta$ is called the \emph{degree}. We again recall the version of Gap Conjecture relevant to us, so let $\beta \in [0, 1)$.

\begin{conjecture}[Conjecture $C^*(\beta)$ of Grigorchuk]
A group either has at most polynomial growth, or at least stretched exponential growth with degree $\beta$.
\end{conjecture}


We next show that in any group with stretched exponential growth in the above sense, for sufficiently large $k$ we can find a subtree that fits in a ball of radius polynomial in $k$, has sufficient branching, and furthermore we can encode this tree in a finite-support configuration of a particular SFT $X$.

We are mainly interested in full shifts, so we will work with SFTs that contain points of finite support, and our first result constructs finite-support points which we will use as markers. This does not require any assumptions on the group. Similar marker constructions can be performed under other (including some strictly weaker) assumptions. See e.g.\ \cite{BaCaRi25,Ho10} for such marker constructions.

\begin{lemma}
\label{lem:Markers}
Let $X$ be a subshift of finite type where some finite-support point has free orbit. Let $C \in \N$ be arbitrary. Then we can find $k$ and a finite family $F$ of points with support contained in $B_k$, such that
\begin{itemize}
\item There is a block map from $X$ to $\{0, 1\}^G$ which maps every $x \in F$ to the characteristic function $\chi$ of $\{1_G\}$.
\item $|F| \geq C |B_{Ck}|^C$.
\end{itemize}
\end{lemma}

\begin{proof}
Let $b_k = |B_k|$ and suppose $G$ has at most $s$ generators. Suppose $x \in X$ has support contained in an $m$-ball, and has free orbit, assume $m$ is a window size for $X$, and assume $m \geq 10$. It is reasonably easy to show that there is a block map $\phi$ which maps $x$ to the characteristic function $\chi$ of $\{1_G\}$. Namely, consider the block map whose local rule identifies the central $B_m$-pattern of $x$ by writing $1$, and otherwise writes $0$. If some shift of $gx$ also has this pattern in $B_m$, then the nonzero symbols of $x|g^{-1}B_m$ are in bijection with those of $x$ in $B_m$. Thus actually $g^{-1}x = x$ because these must contain all the nonzero symbols, thus $g = 1_G$.

We claim that there exists $\alpha > 1$ such that for all $k$ we can find a family $F$ of points with support contained in $B_k$, and with cardinality at least $\alpha^{b_k}$, and such that some block map maps all of them to $\chi$.

To see this, first we may modify the local rule of $\phi$ so that it only outputs $1$ if it sees content from $x$ in the $10m$-ball. Thus, we can use the same block map, if we only use points where there is a copy of the central pattern of $x$ separated by at least $10m$ from other such copies, and these other copies are $3m$-dense in the annulus $B_k \setminus B_{10m}$ (which easily implies that each of them sees another one in the annulus $B_k \setminus B_{10m}$ for large enough $k$).

By the Packing Lemma we can indeed find a subset of the annulus which is $3m$-separated and $3m$-dense, and apart from the boundaries we can freely center copies of left translates of $m$-balls in these positions. Then we can encode information in the point by varying the exact positions by $1$ in various directions (we can move each pattern by at least $m/3$ before any two may intersect). In particular, we have a positive density of positions where we can write a bit of information, which easily implies the claim in the previous paragraph.

Now we claim that if $\alpha > 1$, then $\alpha^{b_k} > C b_{Ck}^C$ for some $k$, which is the claim about the size of $F$. Suppose not. Then for all $k$ we have
\[ \alpha^{b_k} < C b_{Ck}^C. \]
Taking logarithms we have
\[ b_k \log \alpha < \log C + C \log b_{Ck} \leq \log C + C^2 k \log s \]
so $b_k < (\log C + C^2 k \log s) / \log \alpha$. Thus the growth of balls is bounded by a linear function of $k$. It is well-known that this implies that $G$ is virtually $\Z$, and in this case, the claim is clear.
\end{proof}

We now want to formalize the idea that we can interpret trees and other graphs on configurations of an SFT on a group, and furthermore on the nodes of these graphs we can interpret symbols that can freely be changed to others. 


The following is a variant of \cite[Definition~1]{Sa20e}; in that definition, the class of graphs was specialized to be a family of cyclic groups, but additionally a subshift was allowed on these groups. A subshift could be allowed on the graphs here as well, but the definition is already quite complicated, and we only need full shifts in our application.

\begin{definition}
\label{def:Simulation}
Let $X \subset A^G$ be a subshift, $B$ a finite alphabet, and let $\mathcal{G}$ be a set of good $S$-labeled graphs. We say $X$ \emph{simulates} $(B, \mathcal{G})$ if, possibly after applying a topological conjugacy to $X$, the following three paragraphs hold.

\begin{enumerate}
\item To each $a \in A$ is associated a number $m(a) \in \N$ (of internal nodes), a word $w(a) \in B^{m(a)}$ (the simulated symbols), an element $v(a) \in ((G \times \N) \cup \{\bot\})^{S^{\pm} \times m(a)}$ (the offsets), and an additional piece of information $\i(a) \in I$ (for some finite set $I$). Furthermore, $a$ is uniquely determined by the tuple $(m(a), w(a), v(a), \i(a))$. To simplify the following discussion, we identify $A$ directly by its image and assume $A \subset \N \times B^* \times (G \times \N)^* \times I$ (of course, finite).

\item To each $x \in X$ we can associate an $S$-labeled graph by taking nodes $V(X) = \{(g, j) \;|\; g \in G, 0 \leq j < m(x_g)\}$, and for each $(g, j) \in V(x)$ with $v(x_g)_{s, j} = (h, j')$ adding an $s$-labeled edge from $(g, j)$ to $(gh, j')$ (so we must have $(gh, j') \in V(x)$). We require that this graph is good for all $x \in X$.

\item The weakly connected components of the graphs are in $\mathcal{G}$, and furthermore, any $H \in \mathcal{G}$ appears as a weakly connected graphs.

\item We require that the simulated symbols can be chosen arbitrarily. Specifically, suppose $x, y \in A^G$ and $x, y$ only differ in position $1_G$, and $m(x_{1_G}) = m(y_{1_G}), v(x_{1_G}) = v(y_{1_G}), \i(x_{1_G}) = \i(y_{1_G})$ (so the only possible difference is $w(x_{1_G}) \neq w(y_{1_G})$). Then we require that $x \in X \iff y \in X$.
\end{enumerate}
\end{definition}

\begin{lemma}
Let $G$ have at least stretched exponential growth, and let $X \subset A^G$ be any SFT such that some finite-support point has free orbit. Then for any $B$ there exists an arboreous family of complete graphs $\mathcal{G}$ such that $X$ simulates $(B, \mathcal{G})$.
\end{lemma}

\begin{proof}
Let $k$ be as in Lemma~\ref{lem:Markers} for $C = \max(|B|^2, 15)$. Let $D = C b_{Ck}^C$, and let $d$ be such that $\beta d > 1$ where $\beta$ is the degree of stretched exponential growth. Then for some constant $c$,
\[ |B_{n^{2d}}| \geq c e^{(n^{2d})^{\beta}} = c e^{n^{2 \beta d}} \geq (D+1)^{n^2} \]
for large enough $n$ (since $2\beta d > 2$).

Pick inside $B_{n^{2d}+k}$ disjoint balls $B^i = g_i B_k$, so that the set of centers $\{g_i\}$ is $3k$-dense and $3k$-separated. We may assume $g_1 = 1_G$ so that $B^1 = B_k$. Connect each ball $B^i$ to all $B^j$ such that $d(g_i, g_j) \leq 7k$, by a two-directional edge. The resulting graph is connected and in fact the distance from $B^1$ to any other $B^i$ is at most $2n^{2d}$: we can follow a geodesic from $1_G$ to $g_i$ in the group $G$, and the sequence of $B^j$'s whose centers are closest to the geodesic (breaking ties arbitrarily) is connected by the edges of our new graph due to $3k$-denseness, and gives a path $B^1$ to $B^i$.

Since the packing is $3k$-dense, we have at least
\[ |B_{n^{2d}}|/|B_{3k}| \geq (D+1)^{n^2}/|B_{3k}| \geq D^{n^2} \]
such centers for large $n$ (note that to fit balls inside $B_{n^{2d}+k}$, the centers should be in $B_{n^{2d}}$). 

Next, take a spanning tree of our new graph, rooted at $B^1$, and with distance from $B^1$ to any $B^i$ at most $2n^{2d}$. For example, go through spheres of larger and larger radius, and arbitrarily connect (in the tree) every element of the $r$-sphere to an element of the previous sphere to which it is connected in the graph.

Next, associate to each node $B^i$ the length at most $2n^{2d}$ path from the root $B^1$ to it. We can code this path by the sequence of offsets between the centers of the balls $B^j$ on the path (as elements of $G$). This associates to $B^i$ a word of length at most $2n^{2d}$ over as set of size much smaller than $D$. We may then code this as a word over $\hat D = \{1, \dots, D\}$, using a tail over one of the symbols (say we use $D$ for this) to mark that the endpoint has been reached.

So at the moment, we have a set $L$ of words in $\hat D^{2n^{2d}}$, with $|L| \geq D^{n^2}$. Geometrically, this corresponds to a superexponentially large tree of polynomial depth, and we want to first conclude that such a tree must have a subtree where there are at least linearly many branchings along every path from the root to a leaf, and then from this to construct an $(n, \{a, b, c\})$-thicket.

One way to conclude the existence of subtrees that branch a lot is to use the winning set $\hat L$ of $L$, as defined in \cite{Sa21e} (based on \cite{SaTo14c}). This is a set of words of the same length and over the same alphabet that can be associated to $L$, which is \emph{downward-closed}, meaning if $v \in \hat L$ and $u \leq v$ in the sense that $u_i \leq v_i$ for all $i \in [0, 2n^{2d}-1]$, then $u \in \hat L$.

The winning set $\hat L$ encodes ``branching structures of uniformly branching trees'' in $L$, meaning that if $v \in \hat L$, then letting $U = \{u \;|\; u \leq v\}$, there is an assignment $T_v : U \to L$ such that for all $u, u' \in U$, we have
\[ \min \{i \;|\; u_i \neq u'_i\} = \min \{i \;|\; T_v(u)_i \neq T_v(u')_i\}. \]
Thus, geometrically, in the spanning tree $v \in \hat L$ corresponds to a subtree where whenever $v_i \geq 2$, the tree branches in at least two in each branch at depth $i$.

Now the claim about the tree branching linearly many times on each path amounts to showing that we can find $v \in \hat L$ with support size at least $n$. We show that indeed this is immediate from the cardinality of $\hat L$ (which is that of $|L|$), by some simple combinatorics.

Let $s$ be the maximal number of non-$1$ coordinates in a word $v \in \hat L$. Then to describe words in $\hat L$ it suffices to list the $s$ positions of non-$1$ symbols (or a superset of them) and describe their contents. Thus we have
\[ D^{n^2} \leq |L| = |\hat L| \leq D^s \binom{2n^{2d}}{s} \leq D^s (2n^{2d} e / s)^s \]
from Stirling's approximation. Taking logarithms we have
\[ n^2 \log D \leq s (\log D + \log (2n^{2d} e / s)) \]
Since $\log (2n^{2d} e / s)$ grows much slower than $\sqrt{n}$, we even have $s = \Omega(n^{1.5})$ for large $n$, meaning we actually have much more than a linear amount of branching in some subtree.

To get an $(n, \{a, b, c\})$-thicket, we actually need the brachings to be ternary, i.e.\ along every path to a leaf, we should have many nodes with three children. This is not automatic, but by using paths of length at most $2$ of the tree (call these \emph{$2$-jumps}) instead of the original edges, 
it is possible to mimic the usual proof that $\{a, b, c\}^*$ embeds in $\{a, b\}^*$ as follows. First, as some preprocessing, step into a subtree so that every branching is binary, and drop every second branching so that there is a path of length at least $2$ between any consecutive branchings.

In the resulting tree, go down the tree from the root, and at the first branching (let us say to a ``left'' and a ``right'' child), use the left branch as such with label $a$, but in the right branch, have an edge with label $b$ to the original right neighbor, and a $2$-jump with label $c$ to the unique child down the tree from the right neighbor. This gives the node three successors connected by $2$-jumps. In the left branch, proceed to modify the tree recursively. In the right branch, continue down the original path with $2$-jumps to simulate two paths on one actual path. At the next branching, the two paths can merge into the successors, and we can recursively continue on these two paths. This again at most halves the number of branchings in the tree, and now every branching is ternary.

Since $s = \Omega(n^{1.5})$, certainly $s/4 \geq n$, so we have shown the existence of an $(n, \{a, b, c\})$-thicket, the labels being given as in  the previous paragraph. Note that now the branchings are no longer uniform (unlike in the tree given by the winning set argument), but this is not required in the definition of a thicket.

By the choice of $k$, we can now position finite $k$-support patterns, namely central patterns of the set $F$ from Lemma~\ref{lem:Markers}, in the balls $B^i$ of the tree, where $|F| \geq C |B_{Ck}|^C$ for $C = \max(|B|^2, 15)$. Distances in the tree of $B^i$s correspond to elements of $G$ with word norm at most $7k$, so double jumps correspond to elements with word norm at most $14k$, so it suffices to code four elements of $B_{14k}$ in each $B^i$, for which $|F| \geq C |B_{Ck}|^C \geq |B_{14k}|^4$ is plenty. We also still have room to store two symbols from the alphabet $B$, which is possible since $C \geq |B|^2$. Recall that a local rule can detect the positions of these patterns, and they do not overlap (by $3k$-separation), so automorphisms can freely change the $B$-symbol without messing up the encodings.

To fit this more explicitly into Definition~\ref{def:Simulation} above, we can apply a topological conjugacy to $X$ which, when a central pattern of one of the points $F$ is visible (meaning the local rule $\phi$ outputs $1$) and at distance at least $3k$ from other such patterns, writes down the index of this point as a number between $\{1, \ldots, |F|\}$, which we assume disjoint from the alphabet $A$ of $X$, and maps the rest of the ball to a new symbol $\#$.

Clearly writing back the $F$-patterns gives a left inverse for this block map, so the map is injective, and thus the image is an SFT $Y$. In this $Y$, some positions of some configurations have a number from $\{1, \ldots, |F|\}$ that can be freely changed. These are the positions where we will take $m(c) \geq 1$ in the simulated graph.

To get not necessarily complete graphs, we could use $m(c) = 1$ always. However, we want the graphs to be complete, while not removing any subgraphs, thus not removing the $(n,\{a,b,c\})$-thickets whose existence we showed.

For this we again can use the conveyor belt trick: we take instead $m(c) = 2$ always (so there are two simulated nodes in each $B^i$), and we wrap incomplete $s$-labeled paths at the ends into a conveyor belt. This is indeed why we store two symbols in each node instead of one. 
Note also that in the conveyor belt construction, we do not need to store more edges, as the two simulated nodes use the same edges, only flipping which edges is incoming and which is outgoing.

Indeed the $(n,\{a,b,c\})$-thickets still appear on the encoded graphs, on the first simulated nodes where edges go forward.
\end{proof}

Combining the above lemma with Theorem~\ref{thm:AbstractPSPACEComplete} gives:

\begin{theorem}
Let $G$ have at least stretched exponential growth, and let $X \subset A^G$ be any SFT such that some finite-support point has free orbit. Then $\Aut(X)$ has a finitely-generated subgroup with PSPACE-hard word problem.
\end{theorem}

\begin{corollary}
Let $G$ be a finitely-generated group with at least stretched exponential growth. Then $\Aut(A^G)$ has a finitely-generated subgroup with PSPACE-hard word problem whenever $A \geq 2$.
\end{corollary}

Combining the above lemma with Theorem~\ref{thm:AbstractSimplisticGrowthLowerBound} gives:

\begin{theorem}
Let $G$ have at least stretched exponential growth, and let $X \subset A^G$ be any SFT such that some finite-support point has free orbit. Then $\Aut(X)$ has simplistic growth at least $e^{2^{\sqrt[\ell]{n}}}!$ for some $\ell$. In particular, $\Aut(X)$ has a finitely-generated subgroup that does not embed in $\Aut(B^{\Z^d})$ for any $d$.
\end{theorem}

\begin{proof}
The lower bound on simplistic growth is from Theorem~\ref{thm:AbstractSimplisticGrowthLowerBound}, and then we observe that $e^{2^{\sqrt[\ell]{n}}}!$ exceeds the simplistic growth upper bound $e^{n^d}!$ for $\Aut(B^{\Z^d})$.
\end{proof}

\begin{corollary}
Let $G$ be a finitely-generated group with at least stretched exponential growth. Then $\Aut(A^G)$ has simplistic growth at least $e^{2^{\sqrt[\ell]{n}}}!$ for some $\ell$, whenever $A \geq 2$. In particular, it has a finitely-generated subgroup that does not embed in $\Aut(B^{\Z^d})$ for any $d$.
\end{corollary}

We can similarly obtain a lower bound on simplistic growth.

\section{The lamplighter group}

Let $C$ be a finite nontrivial abelian group. The \emph{lamplighter group} is $L = C \wr \Z = \bigoplus_\Z C \rtimes \Z$, where in the latter representation, $\Z$ acts by shifting the coordinates in $\bigoplus_\Z C$.

\begin{theorem}
\label{thm:Lamplighter}
Let $C$ be a finite nontrivial abelian group, let $A$ be a nontrivial finite alphabet, $n \geq 2$, and $L = C \wr \Z$. Then $\Aut(A^G)$ has a finitely-generated subgroup with co-NEXPTIME-hard word problem.
\end{theorem}

\begin{proof}
We will assume $C = \Z_2$, because all the cases are similar (or rather, $\Z_2$ is the hardest case, as it needs a preprocessing step due to insufficient branching). We will also assume that $A$ is of specific form, as in the general case one can argue as in the proof of Theorem~\ref{thm:HochmanSolution}, i.e.\ work with a set of unbordered words representing symbols from a larger alphabet (observing that the lamplighter group has plenty of self-embeddings).

Let $T$ be a universal nondeterministic Turing machine, and recall from Section~\ref{sec:BasicCT} the standard representation of Turing machines by Wang tiles $W$ such that tilings by $W$ (with certain form) correspond to computations of $T$. We refer to Wang tile directions by north, south, west and east. 
It is helpful to make two small modifications to this tile set. 
First when the machine accepts, this is visible on the entire top row of the tiling (when tiling rectangles). This is easily obtained by using a signal on the Wang tiles (in the west-east colors), which is transmitted by the head on each row. Second, there is a tile $\square$ (that can be the blank symbol of the Turing machine) that has $\square$ as its west and east colors, and is the only tile with $\square$ as its west color. This can be used to force a row to end in $\square$s.

Now as the alphabet we take $A = \Z_5 \times \Z_6 \times \Z_7 \times W \times \{0,1,2\}^2 \times B$, where $B = C^3$ and $C$ is the set of colors of Wang tiles.

The basis of the proof is the same as in \cite{BaSa24}. The lamplighter group contains large grids. More specifically, geometrically this group is a horocyclic product of two full binary trees, which means that we can obtain it as a connected component of a direct product of such trees, where one tree is inverted, and we have to move up and down simultaneously in the two trees. (Thus when in one tree a node branches, in the other we move to a parent.)

Now, a full binary tree of depth $n$ naturally has $2^n$ elements on its frontier, and we can identify them with the numbers $[0,2^n-1]$. In the direct product, there is then a natural copy of a square $[0,2^n-1]^2$. The square happens to be in a single connected component of the horocyclic product, which will allow us to use automorphisms to ensure computation on the massive grid obtained by connecting ($\Z^2$-)neighboring elements of the square.

It is useful to step into a subgroup before describing the grids formally. Let $G$ be a subgroup of $L$ isomorphic to $\Z^2_2 \wr \Z$. Write $L'$ for the natural copy $\Z_2 \times \{0\} \wr \Z$ of $L$ in $G$. Let $t$ be the natural generator of $\Z$ in $G$ (i.e.\ the ``shift map''), and write $a = t$, $b = t \cdot (1, 0)$, $c = t \cdot (0, 1)$ where $(1, 0), (0, 1)$ are the natural generators of $\Z_2^2$. So $a, b$ are a standard set of generators of $L'$, and $a, b, c$ one for $G$. The point of this is that for each $g \in L'$, $g + \langle c \rangle$ is a copy (coset) of $\Z$ where we can encode indices, and for distinct $g, g' \in L'$ these cosets are disjoint. We will not need to consider a generating set for $L$, and could work directly in $G$ from now on, but we can equally well work in $L$ directly, independently in each left $G$-coset.

Let $g \in L$. The \emph{grid (of size $2^n$ at $g$)} is the function $f : \Z_{2^n} \times \Z_{2^n} \to gL'$, which for $u = \bin(\ell), v = \bin(m)$ maps
\[ f_g(\ell, m) = g \prod_{\{i \;|\; u_i = 1\}} a^{-i} b a^{i-1} \prod_{\{j \;|\; v_j = 1\}} a^j b^{-1} a^{-j+1}. \]
Write also $f(\ell, m) = f_1(\ell, m)$ so $f_g(\ell, m) = gf(\ell, m)$. Here, words are $1$-indexed and binary representations are written with least significant bit first. Sometimes, we call the image of this function the grid, and $g$ is the \emph{root} of the grid. 

Note that $f(\ell, m)$ for different $\ell, m$ commute as elements of $L'$ (the total shift of such an element is $0$, so they belong in the $\bigoplus_\Z \Z_2$-part of $G = \Z_2^2 \wr \Z = \bigoplus_\Z \Z_2^2 \rtimes \Z$), and form a copy of $\Z_2^{(2^n)^2}$ inside it, whose action can be described (through the identifying map $f$) as swapping adjacent rectangles at different levels. Note that in particular that if $h$ is in the grid of $g$, then the grids of $g$ and $h$ agree as sets, but define different ``orientations'' for it.

When a root is fixed, we think of the grid as being drawn on a horizontal plane $\R^2$ with north, south, west and east as the four directions. The elements $a, b$ move ``up'' (out of the plane) in this point of view, For a picture of a piece of this group (drawn as described here), and a highlighted grid, see \cite{BaSa24}.

The \emph{grid coordinates} of $g \in L$ is the pair of numbers read in the two first components of the $\{0,1,2\}^2$-component of the alphabet, read in positions $g, gc, \ldots, gc^{n-1}$, and interpreted in binary (if there are $2$s present, the grid coordinates are undefined).

The basic idea (which we will have to modify a bit) is to write an element of $\PAut(A)$ of polynomial word norm in $n$, which works as follows: For $g \in G$, it performs a nontrivial $\Z_7$-permutation if and only if on the grid of size $2^n$ at $g$,
\begin{itemize}
\item there is a valid tiling of $W$ in the $W$-components of the alphabet;
\item the grid coordinates of $f_g(\ell, m)$ are $(\ell, m)$;
\item on the southmost row, i.e.\ in $f_g(0, 0), f_g(1, 0), \ldots f_g(2^n-1, 0)$, we have a specific initial configuration (the initial configuration we are reducing to the word problem of $\PAut(A)$) for the Turing machine being simulated by $W$ (with polynomially many initial cells having specified content, and then only padding);
\item on the northmost row we have a configuration representing an accepting computation of $T$.
\end{itemize}

If we can write a formula for such an element in polynomial time, then we have a reduction from the problem of checking whether a length $2^n$-computation of $T$ accepts the given input of size $n$, to the co-word problem of $\Aut(A^L)$. The problem being reduced is indeed NEXPTIME-complete since $T$ is a universal Turing machine, and because Wang tiles simulate the Turing machine explicitly, i.e.\ they directly code a spacetime diagram of the Turing machine with the (south-north) height of the tiling corresponding exactly to computation time, and (west-east) width exactly to space usage. Note also that (as in the case of the PSPACE-hardness proof in Lemma~\ref{lem:PSPACEByTrees}) while we allow only an exponential time rather than $2^{O(n^d)}$ as in the definition of NEXPTIME, we can always produce a polynomial padding in the reduction.

It is difficult to construct an element as described above directly, since the grid is not directly connected inside $L$, and we have to do all the checking outside the grid. Thus, besides cheching the above conditions on the grid, we have auxiliary requirements on connecting trees. Again, we will use ripple catching. 

More specifically, observe that for fixed $i \in [0, 2^n-1]$ we have
\[ f([0, 2^n-1], i) = f(0, i) a^{-n} \{a, b\}^n. \]
Thus, to connect $f(0, i)$ to the entire row $f([0, 2^n-1], i)$, we can move by a right translation by $a^{-n}$, and then simply traverse up the full binary tree $\{a, b\}^n$. Note that in our geometric setting, this connecting takes place below the grid.

Similarly, there is a tree above the grid connecting each column, but we will not give the formulas and argue instead by symmetry, noting that the lamplighter group admits the automorphism $a \mapsto a^{-1}, b \mapsto b^{-1}$ that inverts the point of view.

Just like in the proof of Theorem~\ref{thm:MonoidCase} we read words over $(0, 1)^*2^*$ (playing the role of the \emph{index}) written on $(hc^0, hc^1, \ldots, hc^{n-1})$ for $h \in \{a, b\}^{\leq n}$ in the first $\{0,1,2\}$-component, and construct an automorphism $\psi_1'$ of polynomial norm in $n$ that performs a nontrivial permutation $\pi$ at position $g$ in the $\Z_7$-component, if the positions $g \{a, b\}^n$ encode a valid (west-east row of a) $W$ tiling. 

To implement this with ripple catching, we firstly include in the notion of successfulness of nodes that the index is incremented as in the proof of Theorem~\ref{thm:MonoidCase}, i.e.\ at the root the index is $2^n$, and when moving up the tree, the generators are appended at the first $2$-symbol. In the second $\{0,1,2\}$-component of the alphabet, as we go up the tree we check that the values stay the same throughout the tree (and $2$ does not appear). All in all, once we land on the grid, the index is form $\{0,1\}^n$ and the second components also form a word in $\{0,1\}^n$, thus we have valid grid coordinates.

For the verification of correctness of the tiling, we have in each node of the tree just a single $B$-symbol, which codes $C^3$ where $C$ is the set of colors of $W$. When going up the tree, if at some node we have the triple $(c, c'', c')$, then in the successor relations we verify that we have $(c, c''', c'')$ in the $a$-branch and $(c'', c'''', c)$ in the $b$-branch, for some choices of $c''', c''''$.

The interpretation of $(c, c'', c')$ in a node is that $c$ (resp.\ $c'$) is the west (resp.\ east) color of the leftmost leaf eventually produced from that node. The role of the middle color $c''$ is technical: in our definition of successfulness through relations, we cannot directly synchronize the same $c''$ for the $a$- and $b$-successor (of course, we could allow this by slightly changing the formalism).

The effect of this is to make the west and east colors consistent on the leaves. On the final $n$th level (which is where the grid of interest is), we check that the colors $c, c'$ of the triple $(c, c'', c')$ are the left and right color of the Wang tile present in that node (it does not matter what the $W$-components contain outside the grid, or what the $c''$-color is).

We should also ensure at this point that the final row has an accepting configuration. This is easy to include in the check performed in the nodes that land on grid coordinates.

The automorphism $\psi_1'$ now performs some permutation ($\pi$) $n$ steps below valid grid rows whenever the grid coordinates and Wang tiles are consistent on that row, and the tree below has suitable content. We want to ``move'' these permutations to the left border on the grid, i.e.\ by $a^n$. We will move them to the $\Z_6$-component: use a partial shift of the $\Z_7$-track to move the (possibly) permuted cell of $\psi_1'$ by $a^n$, conditionally permute the $\Z_6$-track, and shift back. Doing this twice, analogously to Equation~\ref{eq:eq}, we can detect nontrivial permutations of $\Z_7$ and convert them to a nontrivial permutation of $\Z_6$ under the same conditions (but now happening on the grid).


Note that for this to work, it is important that the $\Z_6$-content is not read or affected by $\psi'$ (which is the case); and that the $\Z_7$-cells on a valid grid are never affected by $\psi_1'$. The latter holds because when $\psi_1'$ acts nontrivially, counters on the grid contain binary numbers in the first counter, while any cell permuted by $\psi_1'$ has has $2$ in its first $\{0,1,2\}$-component (being the root of a connecting tree, thus having index $2^n$).

From this we obtain an automorphism $\psi_1$ that performs a nontrivial permutation on the $\Z_6$-component at $g$ if the tiles $f_g([0, 2^n-1], 0)$ match (west-east) horizontally, the counter values are consistent in the sense that the value at $f_g(i, 0)$ is $(i, j)$ (and $j$ stays constant on the row), the nodes on the tree below have correct indices, and $\Z_7$-data (all zero, except the root having some element from the support of $\pi$).

By symmetrically using trees that connect the grid from ``above'', we have elements $\psi_2$ that only act nontrivially at $g$, in the $\Z_6$-component, if the ``column'' $f_g(0, j)$ has consistent colors in the Wang tiles, in the south and north directions, and the counters are also consistent.  

In the argument giving $\psi_1$, the basis is that we can perform a controlled permutation on the $\Z_7$-component based on coordinates being successful (though here we only implicitly defined the notion of successfulness). These formulas are ultimately commutator formulas, and the successfulness checks are ultimately of the form ``perform a permutation $\pi'$ if the node is successful''. Thus, we can equally well use the map $\psi_1$ in place of the notion of successfulness, but now using the $\Z_6$-components, and use the trees above the grid connecting columns from above, analogously to the definition of $\psi_2$.

Then we can use partial shifts (as in the transition from $\psi_1'$ to $\psi_1$) to move this permutation to the $\Z_5$-component of a node on the grid. This gives us an automorphism $\psi_3$ that performs a nontrivial permutation in the $\Z_5$ component at $g$ if and only if every west-east row of the grid $f_g([0,2^n-1], [0,2^n-1])$ rooted at $g$ is consistent (individually). Symmetrically from $\psi_2$ we can build, analogously to the definition of $\psi_1$, an automorphism $\psi_4$ that acts nontrivially if and only if all the columns are valid.

Now the commutator $[\psi_3, \psi_4]$ has word norm polynomial in $n$ in $\PAut(A)$, and performs a nontrivial permutation at $g$ if and only if every west-east row and every south-north column of the grid $f_g([0,2^n-1], [0,2^n-1])$ has consistent tiling and consistent Wang tiles.

Finally, by taking another commutator, we can additionally condition the permutation by the initial configuration. Using the tile $\square$, we can in NC$^1$ force the exponentially long tail to be $\square$ after the input.

Thus, we have reduced the NEXPTIME-complete problem of existence of $2^n$-time $T$-computations on polynomial-sized inputs (through the existence of tilings) to the complement of the word problem of $\PAut(A)$.
\end{proof}

\section{Upper bounds and polynomial splitting schemes}

In this section we show the co-NP, PSPACE and co-NEXPTIME upper bounds. In the case of co-NP, although it does not thematically fit, completeness is also shown in this section; the reason is that the proof requires only a single sentence.

\subsection{The co-NP upper bound (and lower bound)}

\begin{theorem}
\label{thm:Polynomial}
For an f.g.\ group $G$ with polynomial growth, the automorphism group of an subshift with NP language has co-NP word problem.
\end{theorem}

\begin{proof}
Polynomial growth is equivalent to virtually nilpotent, by Gromov's theorem. It is known that these groups have polynomial-time word problem. The word problem of any f.g.\ subgroup can be solved by guessing a witness (which fits in a polynomial-sized ball due to the growth rate). Checking the witness is in the language of the subshift is in NP by assumption, and checking that the automorphism acts nontrivially on it can be done in polynomial time by applying the local rules.
\end{proof}

\begin{corollary}
\label{cor:Polynomial}
For an f.g.\ infinite group $G$ with polynomial growth, $\Aut(A^G)$ has a finitely-generated subgroup with co-NP-complete word problem.
\end{corollary}

\begin{proof}
By the previous theorem, the word problem is in co-NP, and since virtually nilpotent groups contain copies of $\Z$, the hardness result for $\Z$ shown in \cite{Sa20e} implies the same for $G$.
\end{proof}

\subsection{The PSPACE upper bound}
\label{sec:PSPACEUpperBound}

We next state a theorem that implies PSPACE upper bounds for word problems of automorphism groups of full shifts in the case of free groups and surface groups. The idea is that in these groups, when looking for a witness for the co-word problem (i.e.\ a pattern where the given element acts nontrivially), we can split the ball of the group in two pieces of roughly equal size, and continue recursively on the two halves. In the case of a free group, we can use the fact that the Cayley graph is a tree (so we can use constant-size cut sets), and in the case of surface groups, we can use geodesics to obtain linear cuts. To make the idea precise and to (perhaps) allow applying it to other groups, we give a general definition of polynomial splitting schemes.

Let $S$ be a fixed symmetric generating set containing $1_G$ for the finitely-generated group $G$.

\begin{definition}
A set $A \subset G$ admits an \emph{$\alpha$-splitting scheme by $k$-cuts} if $A = \emptyset$, or there is a set $C$ of cardinality at most $k$ such that $A \setminus C = A_1 \cup A_2$ such that $A_1S \cap A_2S = \emptyset$, $|A_1| \leq \alpha|A|$, $|A_2| \leq \alpha|A|$ and each $A_i$ admits an $\alpha$-splitting scheme by $k$-cuts. We say $G$ admits polynomial splitting schemes if there exists $d$ and $\alpha < 1$ such that $B_n$ admits an $\alpha$-splitting scheme by $(n^d+d)$-cuts for all $n$.
\end{definition}

It is easy to see that this does not depend on the generating set (by thickening the splitting sets defined with respect to one generator), more generally see Lemma~\ref{lem:QI}.

\begin{theorem}
\label{thm:PSPACEUpperBound}
Let $G$ be an f.g.\ group with PSPACE word problem. If $G$ admits a polynomial splitting scheme, the automorphism group of any PEPB subshift has PSPACE word problem.
\end{theorem}

\begin{proof}
We first consider a full shift $A^G$, and then explain how to extend this to PEPB subshifts. It suffices to show that the co-word problem is in alternating polynomial time AP, since co-PSPACE = PSPACE = AP (the first equality is trivial from the determinism of the PSPACE model, and the second is Lemma~\ref{lem:AP}).

Let $F$ be a finite set of automorphisms, and let $r$ be a bound on their radii. We may assume $r = 1$ by changing the generating set $S$ of $G$. Given a word $w$ of length $k$ over $F$, we now describe how to prove the nontriviality, equivalently that there exists a pattern $p \in A^{B_k}$ such that composing the local rules of the cellular automata $w$ in sequence and applying them to $p$ gives a symbol distinct from $p_{1_G}$ at $1_G$.

The witness can be taken to be a pattern in $A^{B_k \times [0, k]}$ such that the patterns $c_i = (g \mapsto A^{(g, i)}) \in A^{B_k}$ are obtained by successively applying the generators in the word $w$ to $c_0$ (and the symbols on the border not given by the local rule are filled arbitrarily). The contents of $B_{k-i}$ in $c_i$ are meaningful, while the rest is more or less meaningless, but makes indexing a little easier. At the end we should check that the contents of $1_G$ in $c_k$ differ from the contents of $1_G$ in $c_0$.

We show that we can figure out in alternating polynomial time whether such a pattern exists. To do this, we guess a polynomial sized subset $C_1$ of $B_k$. Then we compute its (polynomial-sized) boundary $C_1S \setminus C_1$. Note that we can check in polynomial space, and thus alternating polynomial time, whether two elements of the boundary are connected, and thus we can figure out a polynomial-size set of representatives of the boundary (or we can use $C_1S \setminus C_1$ directly and redo much of the computation).\footnote{There is a little subtlety here: In alternating polynomial time computations, we may use PSPACE computations in recursive calls, which means we have alternating access to PSPACE. It may seem that this implies that APTIME trivially contains APSPACE, which would be surprising since APSPACE = EXPTIME and EXPTIME $\overset{?}=$ PSPACE = APTIME is a major open problem. The point is that alternating access from a polynomial time computation to PSPACE does not mean we have access to APSPACE, as in the latter we can do exponentially long computation including alternations, only ``around'' these recursive calls.} We also guess a pattern $p \in A^{C_1 \times [0, k]}$ at this point.

The point is that we want to now check for each boundary element $g$, whether, letting $D$ be the connected component of the Cayley graph of $B_k \setminus C_1$ containing $g$, we can find a pattern $q \in A^{D \times [0,k]}$ which agrees with $p$ in $C_1 \times [0,k]$, everywhere follows the local rules (in the sense described two paragraphs above), and that the contents of $1_G$ in $c_k$ differ from the contents of $1_G$ in $c_0$ if $1_G \in D$. The algorithm may not be correct for all $C_1, g$,\footnote{In case we are not following a good sequence of splits, we may not shrink the set quickly enough, and the algorithm itself cannot easily check how much it shrinks the set at each point.} but its positive claims are always going to be correct.

To do this, we recursively check for each component $D$ of $B_k \setminus C_1$ (represented by $C_1$ and an element of $C_1S \setminus C_1$) that we can find a polynomial-sized cut $C_2$ of $D$ which is contained in that component (again, use polynomially many PSPACE computations to ensure that $C_2$ is indeed contained in $D$, by finding paths to the boundary element $g$). When stepping forward to a connected component, we guess an extension of the previous pattern $p$ to $C_2$, and check that the local rules are followed everywhere.

Continuing this way, if a polynomial splitting scheme exists, in about $\log_\alpha |B_k|$ steps of recursion (which amounts to a linear amount of steps in $k$ since $\alpha > 1$) we have figured out whether a witness exists: If one exists, we can use existential steps to ensure that the $C_i$ follow the polynomial splitting scheme, and that the patterns guessed on the $C_i$ come correspond to a pattern where $w$ acts nontrivially; when stepping to connected components, we use universal steps to ensure that all of the components admit such valid contents. On the other hand, the algorithm only makes correct positive claims whether or not polynomial splitting schemes exist.

To extend the result to PEPB subshifts, we find a locally valid pattern on $B_{n^t}$ where $t$ is the PEPB degree, and the algorithm in all branches checks local validity of the pattern $c^0$. Note that $((n^t)^d+d)$ is still polynomial in $n$, so the polynomial splitting scheme for the $n^t$-ball still has cuts of polynomial size. By the definition of PEPB, the central $B_k$-part is globally valid, so the algorithm correctly identifies valid patterns where automorphisms act nontrivially.
\end{proof}

Note that it is not enough in the above proof to assume that the subshift has PSPACE word problem, since we do not at any point directly check inclusion in the language, but only look locally for occurrences of finitely many forbidden patterns.

The main examples we have of groups with polynomial splitting schemes are free groups and surface groups.

\begin{lemma}
Free groups have polynomial splitting schemes.
\end{lemma}

\begin{proof}
For a free group, the Cayley graph is a tree, and as the cuts one can use successive elements of geodesics: at each step, we have a single branch of the ball rooted on some $g$, and stepping forward by one step in the unique geodesic from $1_G$ to $g$ splits the branch in three sets of equal size, which gives a polynomial splitting scheme with $\alpha = 2/3$.
\end{proof}

In the case of surface groups, we assume that the reader is familiar with some basic notions from geometric group theory. A \emph{surface group} is the fundamental group of a closed surface with some genus $g$. In the case $g = 1$, the group is $\Z^2$ and word problems in automorphism groups are in co-NP. The surface groups with $g \geq 2$ are known to be commensurable, and nonamenable. For $g = 0$ the group is trivial and $g < 0$ reduces to $g > 0$ after taking a finite cover (giving a commensurable group).

Thus, it suffices to consider the genus $2$ surface group with presentation $\langle a_1, b_1, a_2, b_2 \;|\; [a_1, b_1][a_2, b_2] \rangle$.

\begin{definition}
The \emph{pentagon model} of the hyperbolic plane is: the graph with vertices $\Z^2$, and undirected edges $\{(x, y), (x+1,y)\}$, $\{(x, y), (2x, y+1)\}, \{(x, y), (2x+1, y+1)\}$.
\end{definition}

\begin{lemma}
\label{lem:Surface}
All nonamenable surface groups are quasi-isometric to the pentagon model.
\end{lemma}

\begin{proof}
Nonamenable surface groups are quasi-isometric to the hyperbolic plane \cite{Ha00}. Specifically, it is classical that the universal cover of the surface is the hyperbolic plane $H^2$, and by \v{S}varc-Milnor the fundamental group is quasi-isometric to it. On the other hand, the hyperbolic plane is well-known to be quasi-isometric to the pentagon model.
\end{proof}

It should be clear how to generalize polynomial splitting schemes to any undirected graph, using the standard notion of cutsets (of vertices) in graphs.

\begin{lemma}
\label{lem:QI}
The existence of polynomial splitting schemes in a graph is invariant under quasi-isometry.
\end{lemma}

\begin{proof}
Balls of one graph are sandwiched between balls of the other graph by the quasi-isometry, so the quasi-isometry maps polynomial splitting schemes of a ball to polynomial splitting schemes of some shape that contains a large ball of some smaller radius. By restricting it to the ball, we obtain a polynomial splitting scheme for it. The radii change only by multiplication with the QI constants, so being of polynomial size retains its meaning.
\end{proof}

\begin{lemma}
Surface groups have polynomial splitting schemes.
\end{lemma}

\begin{proof}
It suffices to show that these exist in the pentagon model. Furthermore, it is easy to reduce the problem to finding a polynomial splitting scheme for ``one-sided balls'' $P_k = \{(x, y) \;|\; y \in [0, k], x \in [0, 2^y-1]\}$, i.e.\ vertices reachable from $(0, 0)$ by only moves that increment the second coordinate. Namely, these sets contain balls of the pentagon model, of roughly the same radius.

For the set $P_k$, a polynomial splitting scheme is obtained from geodesics between $(0,0)$ and vertices of the form $(x, k)$ with $x \in [0, 2^y-1]$. Specifically, identify $(x, k)$ with $x/2^k$ and repeatedly cut the real interval in half.
\end{proof}

\begin{lemma}
If $G$ has a polynomial splitting scheme, so does $G \times H$ whenever $H$ has polynomial growth.
\end{lemma}

\begin{proof}
Taking $S \times T$ as a generating set (which by Lemma~\ref{lem:QI} does not affect the existence of polynomial splitting schemes), a ball of $G \times H$ is a product of balls of $G$ and $H$. Balls $B_k^H$ are polynomial size in $H$, so the cutsets $C$ for balls of $G$ appearing in the definition of a polynomial splitting schemes can simply be replaced by $C \times B_k^H$.
\end{proof}

\begin{theorem}
For free groups, surface groups and their products with virtually nilpotent groups, the automorphism group of any PEPB subshift has PSPACE word problem. Furthermore, in the nonamenable case, automorphism groups of full shifts have f.g.\ subgroups whose word problems are PSPACE hard.
\end{theorem}

\begin{proof}
These groups have polynomial splitting schemes, so the word problem of the automorphism group of any PEPB subshift is in PSPACE by Theorem~\ref{thm:PSPACEUpperBound}. The second claim follows from Theorem~\ref{thm:MonoidCase}, since nonamenable free groups and surface groups have free submonoids.
\end{proof}

\subsection{The co-NEXPTIME upper bound}

We now give the upper bound for the word problem that ``always'' holds; of course, we need some technical assumptions on the complexity of the group $G$ itself, and the complexity of the subshift.

\begin{theorem}
\label{thm:CoNEXPTIMEUpperBound}
Let $G$ be an f.g.\ group. The automorphism group of any subshift $X \subset A^G$ has co-NEXPTIME word problem under either of the following conditions:
\begin{itemize}
\item $G$ has co-NEXPTIME word problem and $X$ has PEPB;
\item $G$ has EXPTIME word problem and $X$ has NP language.
\end{itemize}
\end{theorem}

In particular, if $G$ has word problem in P (as natural groups tend to do), then it has EXPTIME word problem. EXPTIME is known to be strictly larger than P, and there are also groups with word problem in EXPTIME $\setminus$ P.


\begin{proof}
We start with the first item, i.e.\ assume $G$ has co-NEXPTIME word problem and $X$ has PEPB with degree $t$. Let $S$ be the generating set of $G$. We show that the co-word problem (nontriviality of a given element) is in NEXPTIME. Given $w$ over a generating set of automorphisms with radius $r$, guess a good $S$-labeled directed graph $H$ (which is supposed to be the actual ball $B_{(nr)^t}(1)$ in $G$). Check (deterministically) that it is equal to the $(nr)^t$-ball $B_{(nr)^t}(h_0)$ around some $h_0 \in H$ in $H$.

For every pair of distinct elements $h, h' \in B_{(nr)^t}(h_0)$ in the graph $H$, check that $h^{-1} h'$ is nontrivial in $G$. This check can be done in NEXPTIME for each individual pair (it takes at most $e^{O(((nr)^t)^d)} = e^{O(n^{td})}$ steps for some $d$), and since there are at most $|S|^{{nr}^t}$ pairs it can be done in $|S|^{(nr)^t} e^{O(n^{td})} = e^{O(n^{t(d+1)})}$ time in total. If the check succeeds, we do not necessarily know that $H$ is equal to the actual ball $B_{(nr)^t}$, but we can be sure that $B_{(nr)^t}$ has the graph $H$ as an image under an edge label preserving graph homomorphism (or in more group-theoretic terms, $B_{(nr)^t}$ covers $H$). 

Now guess a pattern $p \in A^H$. Check that it has no forbidden patterns (i.e.\ that it is locally valid), and that applying the local rules of the automorphisms in $w$ produce a difference at the origin $h$. If such $p$ exists, $w$ is nontrivial, as we can lift $p$ to $q \in A^{B_{(nr)^t}}$ without introducing forbidden patterns, and then $w$ acts nontrivially on the central $B_{nr}$-pattern of $q$, which by PEPB is a pattern appearing in $X$. On the other hand, if $w$ is nontrivial, then the algorithm succeeds in procing nontriviality since we can use $H = B_{(nr)^t}$ and any pattern $p$ where $w$ acts nontrivially.

Now suppose $G$ has EXPTIME word problem and $X$ has NP language. In this case, we can actually compute $B_{nr}$, and we can check whether a guessed pattern $p \in A^{B_{nr}}$ is in the language of $X$ by giving it as an input to the assumed NP algorithm. Since the pattern is of exponential size, the runtime is at most polynomial in $O(e^n)$, thus we can do it in NEXPTIME.
\end{proof}

\section{Open questions}

While our approach shows that $A^{\Z^D}$ does not embed in $B^{\Z^d}$ for $D > d$, it seems not to be helpful for comparing $A^{\Z^d}$ and $B^{\Z^d}$ for distinct alphabets. (This is another embeddability question of Hochman.)

Also, while we obtain several nonembeddability results (and we ask several technical results concerning the extensions of these results), the following general problem stays open:

\begin{question}
Are there nontrivial embeddings of $\Aut(A^G)$ into $\Aut(B^H)$, for some nontrivial alphabets $A, B$, and f.g.\ infinite groups $G, H$, such that $G$ is not a free group?
\end{question}

Of course, embeddings exist when $G$ is a subgroup of $H$, and $|A| \leq |B|$. However, we do not know any other examples.

We showed in Theorem~\ref{thm:Lamplighter} that the automorphism group of a full shift on the lamplighter group has co-NEXPTIME complete word problem, which as we showed in Theorem~\ref{thm:CoNEXPTIMEUpperBound} is in a sense the maximal possible. It is an interesting question whether we can reach this upper bound in a large class of groups. 

\begin{question}
For which groups $G$ does $\Aut(A^G)$ have co-NEXPTIME-complete word problem?
\end{question}

Some interesting cases to look at would be polycyclic groups (that are not virtually cyclic), hyperbolic groups and $F_2 \times F_2$.

If one can Lipschitz embed a copy of the lampligher group in a group $G$, then one should be able to obtain the hardness result with the usual method of Section~\ref{sec:Arboreous}, at least with some alphabets. However, at present we do not know how common this is.

In any case, it seems likely that this method is suboptimal, and one should instead consider a generalized notion of ``biarboreousness'' (the lamplighter group is a horocyclic product of trees, and this is precisely what allows us to perform computation on an exponential sized grid), allowing for contracting paths instead of strict trees. But again, we do not know how commonly such structures are found in groups.

There are many questions left about simplistic growth. We showed in Theorem~\ref{thm:SimplisticLowerBoundGeneral} that for groups with stretched exponential growth, one has a certain lower bound (far surpassing the corresponding upper bound on free abelian groups) on simplistic growth. This shows that such groups do not embed in $\Aut(A^{\Z^d})$ (or other groups with polynomial systolic growth). However, we do not immediately obtain nonembeddability results between other $\Aut(A^G)$ groups, among groups $G$ of stretched exponential growth.

By optimizing the degree $\ell$, which relates to both the growth rate of the group and the efficiency of implementation of the scheme, we would possibly allow proving many more noninclusions between $\Aut(A^G)$ groups. However, at present it is not clear what the optimal approach is.

\begin{question}
Can we in general attain simplistic growth in $\Aut(A^G)$, which is closely driven by the growth $f$ of the group $G$? For example, can we always (or in a large class of groups) reach $e^{f(n)/\poly\log(n)}!$ simplistic growth?
\end{question}

In the case of groups of polynomial growth (equivalently, virtually nilpotent groups), it seems likely that one can use normal forms for group elements, and implement a variant of Lemma~\ref{lem:LowerBoundTechnical} to always attain roughly the (factorial of the exponential of the) growth rate, as we do in the case of abelian groups.

We note, however, that this approach is fundamentally limited, and the maximal simplistic growths we can attain with such an approach are of the form $e^{2^{\sqrt[d]{n}}}!$. It is certainly possible to have much higher simplistic growths in some groups. For example, if the group $G$ contains an element with prime order $p$ and word norm $n$, then the corresponding right shift shows that the simplistic growth satisfies $f(n) = p$.

There is in fact \emph{no limit} to how quickly elements with high (but finite) prime order can be reached in $2$-generated groups, by (the proof of) the embedding theorem in \cite{HiNeNe49}. Specifically, this proof shows that for each $n$, we can freely choose the order of a single element with word norm $O(n)$. In the residually finite case, we have an upper bound on simplistic growth in terms of the systolic growth, so it would be natural to try to show that the systolic growth (rather than the growth of balls) can be reached, or at least approached, in a large class of groups.

\begin{question}
Can we in general attain simplistic growth in $\Aut(A^G)$, which is closely driven by the systolic growth of the group $G$?
\end{question}

Finally, the technique of Section~\ref{sec:PSPACEUpperBound} is based on ``polynomial splitting schemes'', whose existence we show for free and surface groups (and slightly beyond). 

\begin{question}
Which groups admit polynomial splitting schemes?
\end{question}

We leave it as an open question what groups admit such schemes. One interesting question is whether they exist for small cancellation groups or hyperbolic groups with one-dimensional boundary.

\bibliographystyle{plain}
\bibliography{../../../bib/bib}{}

\newpage
\appendix

\section{Index}
\label{sec:Index}

\begin{tabular}{ll}
$A$ & an alphabet, often $B \times C$ in case of $\PAut$ \\
$\Aut(X)$ & automorphism group of a subshift $X$ \\
$B$ & the control part of an alphabet (for controlled permutations), or coordinate part (for ripple catching) \\
$\beta_{\pi, \ell}$ & applies $\pi$ on the $C$-component of the color of a node $u$ if the $B$-component of $u$ contains $b$ such that $\rho(b) = \ell$. \\
$C$ & the permuted part of an alphabet \\
$d, D$ & dimensions; or $d$ can also be a distance \\
$\eta_\pi \in \PAut(A)$ & Symbol permutation $\pi \in \Sym(A)$ \\
$f, g$ & functions \\
$G$ & the shift group, acting on $A^G$ \\
$\gamma_{\pi}$ & apply $\pi$ at a node (in the $C$-component) if that node is successful \\
$\mathcal{G}, \mathcal{H}$ & a subgroup of $\Aut(X)$ \\
$\mathcal{G}_{m, C, B}$ & the group of controlled permutations  \\
$\mathcal{G}_{S, C, B}$ & the ripple group \\
$i, j, k, m, n, p$ & numbers, various uses \\
$\mathcal{Q}_{S, C, B}$ & good $S$-labeled $B \times C$ colored graphs  \\
$\PAut(A), \PAut[B; C]$ & groups generated by partial shifts and symbol permutations \\
$\pi$ & permutation, usually $\pi \in \Sym(C)$ \\
$\pi_{b,i} \in \mathcal{G}_{m,C,B}$ & the generator that applies $\pi$ if the $i$th bit of the $B$ component contains $b$ \\
$\phi_{\pi, c} \in \mathcal{G}_{S, C, B}$ & apply $\pi$ in $C$-component of node $u$ if $S$-neighbors of $u$ (that exist) have symbol $c$ \\
$\phi_{\pi, c, s} \in \PAut(A)$ & $\phi_{\pi, c}$-analog (for a single $s \in S$) in automorphism group \\
$\phi_{\pi, c, S} \in \PAut(A)$ & $\phi_{\pi, c}$-analog in automorphism group \\
$\psi_{\pi, \ell, c}$ & apply $\pi$ in $C$-component of node $u$ if rank of $u$ is $\ell$ and $S$-neighbors of $u$ have symbol $c$ \\
$\rho$ & rank function of a good ranked graph \\
$\sigma_{g, i}$ & partial (right!) shift by $g$ on track $i$ \\
$s$ & edge-label, or generator of a free group \\
$S$ & finite set of edge-labels, or generators of a free group; often acts on labels \\
$t$ & word over edge-labels \\
$u, v$ & nodes of an abstract graph; or finite words \\
$\vec{u}, \vec{v}$ & vectors in $\Z^d$ or $\Z^D$ \\
$w$ & finite word \\
$x, y, z$ & configurations in a subshifts \\
\end{tabular}

\section{Complexity theory}
\label{sec:CT}

\subsection{Basic complexity theoretic definitions}
\label{sec:BasicCT}

In this section, we give standard complexity theoretic definitions of Turing machines, P, NP, PSPACE, NEXPSPACE and NC$^1$, as well as some terminology we use with them. 

A \emph{complexity class} is a set of languages $L \subset \Sigma^*$ for finite alphabets $\Sigma$.

A \emph{nondeterministic Turing machine} $T$ is a tuple $(Q, \Sigma, \Delta, q_{\mathrm{init}}, F, \square)$, where $Q$ is a finite set (of \emph{states}), $\Sigma$ is an alphabet (the \emph{tape alphabet}), $F \subset Q$ is the set of \emph{final states}, $q_{\mathrm{init}}$ is the \emph{initial state}, $\Delta \subset Q\times \Sigma \times Q \times \Sigma \times \{-1, 0, 1\}$, and $\square \in \Sigma$ is the \emph{blank symbol}. An \emph{instantaneous description} or \emph{ID} is any element of $(\Sigma \cup (Q \times \Sigma))^+$ or $(\Sigma \cup (Q \times \Sigma))^\omega$ where exactly one symbol from $Q \times \Sigma$ appears. In the case of $\omega$, we require all but finitely many symbols are $\square$. The \emph{head position} of an ID $w$ is the length of the unique word $u$ such that $w = u(q, s) v$ for some $q \in Q, s \in \Sigma$, and $q$ is the \emph{head state}. It is sometimes useful to use another coding of IDs where we use words in $\Sigma^*Q \Sigma^+$, by using a word from $Q\Sigma$ in place of the $(Q, \Sigma)$-symbol.

An instantaneous description is \emph{initial} if the head position is $0$, and \emph{final} if the head state is in $F$. Next we define \emph{computation steps (of type $(q, s, q', s', d) \in \Delta$)}. If $(q, s, q', s', 0) \in \Delta$, then there is a computation step (of this type) from $u(q, s)v$ to $u(q', s')v$ for all (possibly infinite) words $u, v$ over $\Sigma$. If $(q, s, q', s', 1) \in \Delta$, then there is a computation step from $u(q, s)av$ to $us'(q', a)v$ for all $u, v$, and any symbol $a \in \Sigma$. If $(q, s, q', s', -1) \in \Delta$, then there is a transition from $ua(q, s)v$ to $u(q', a)s'v$ for all $u, v$, and any symbol $a \in \Sigma$. A \emph{computation} is a sequence of IDs where all adjacent IDs are computation steps. A computation is \emph{finite} if it has finitely many IDs. The \emph{length} of a finite computation is the number of IDs in it.

We say $T$ is \emph{deterministic} if the implication $(q, s, q', s', d), (q, s, q'', s'', d') \in \Delta \implies q' = q'', s' = s'', d = d'$ holds for all $(q, s) \in Q \times S$. For $w = av \in \Sigma^*$ (with $a \in \Sigma$) write $q_{\init} w = (q_{\init}, a) v$, and $q_{\init}\epsilon = q_{\init}\square$. We say $T$ \emph{accepts} $w \in \Sigma^*$ \emph{in time $x \in \N$ and space $y \in \N$} if there is a computation whose initial ID is $q_{\init}w\square^{y - |w|}$ and length at most $x$. We may omit the time or space requirement separately.

Another less obvious coding we use is by Wang tiles. Specifically, there is a standard way to construct a finite set of Wang tiles $W \subset C^{\{\mathrm{north}, \mathrm{south}, \mathrm{west}, \mathrm{east}\}}$ from a Turing machine, so that tilings of rectangles (satisfying the usual Wang tiling rule that the left (west) color of a tile must match the right (east) color of its left (west neighbor)) correspond to computations of the Turing machine. We omit the detailed construction (see e.g.\ \cite{Ro71}), but the idea is simply to use arrows to transmit the Turing machine head correctly (the colors ensure that the arrows match, and the head moves to the correct tile.

The complexity class P (resp.\ NP) contains the languages $L$ which are accepted by a deterministic (resp.\ nondeterministic) Turing machine in polynomial time. Specifically, $L \subset \Sigma^*$ is in NP if there exists a Turing machine $T$ and $d$ such that for all $w \in \Sigma^*$, $T$ accepts $w$ in polynomial time. The class PSPACE contains the languages accepted by a deterministic Turing machine with polynomial space (by a theorem of Savitch, one may equivalently use nondeterministic Turing machines). NEXPTIME contains the languages accepted in time $2^{O(n^d)}$ for some $d$. 

If $\mathcal{C}$ is a complexity class, co-$\mathcal{C}$ is the set of complements of languages in $\mathcal{C}$. We say a language $L \subset A^*$ is \emph{$\mathcal{C}$-hard} if for each $L' \in \mathcal{C}$ with $L' \subset B^*$, there is a deterministic Turing machine which in polynomial time computes (in the sense that after it enters a final state, some specific track of its output contains the output word) from $w \in B^*$ a word $u \in A^*$ such that $u \in L$ if and only if $w \in L'$. We say $L$ is $\mathcal{C}$-complete if it is in $\mathcal{C}$ and $\mathcal{C}$-hard.

A \emph{circuit} is a directed graph with vertices $V(G)$ and edges $E(G) \subset V(G)^2$ which is acyclic (meaning no forward cycles), such that each vertex a label from $\{\i, o, \wedge, \vee, \neg\}$, a vertex with label from $\{\wedge, \vee\}$ has in-degree $2$, a vertex with label $\i$ has in-degree $0$, and others have in-degree $1$. Furthermore, vertices with label $o$ have out-degree $0$. If there are $m$ vertices with label $\i$ and $n$ with label $o$, the circuit computes a function $\{0,1\}^m \to \{0,1\}^n$ in an obvious way, by (w.r.t.\ some ordering of vertices) writes the bits on the vertices with label $\i$, computes a bit on the $\{\wedge, \vee, \neg\}$-nodes by using the labeling Boolean function on the input bits, and finally reading off the bits on the $o$-nodes. The \emph{depth} of a circuit is the maximal length of a forward path from an $\i$-node to an $o$-node.

Now, \emph{(nonuniform) NC$^1$} (where NC stands for ``Nick's class'') is the class of languages $L \subset \{0,1\}^*$ such that for some $C$, for each $n$, there is a circuit of depth $C \log n$ that computes the characteristic function of $L \cap \{0,1\}^n$. (Often, polynomial size is also required, since it is relevant in many other circuit-based complexity classes, but in the present case, this is automatic from logarithmic depth.)

\begin{lemma}
In the definition of NC$^1$, one may assume only $\i$-labeled vertices of the circuit may have out-degree greater than $1$.
\end{lemma}

\begin{proof}
The only requirement is logarithmic depth, so we may make as many copies of the nodes as we well, and there is no need to use the same node twice. Specifically, start from the output (nodes labeled by $o$), and traverse backward to the input, unfolding the acyclic graph to a tree (apart from incoming edges from the input nodes labeled by $\i$).
\end{proof}

One may equivalently think of NC$^1$ as the languages where for each length $n$, there is a Boolean formula describing the words of length $\{0,1\}^n$ (in terms of $n$ basic Boolean variables), and which has depth $O(\log n)$. Internal nodes with out-degree greater than $1$ correspond to variables being reused in this point of view, but as shown this is irrelevant for the strength of the model.

We recall a weak version of the \emph{Master Theorem} from algorithm analysis: if a function $T(n)$ satisfies $T(n) \leq aT(n/b) + f(n)$, then writing $c = \log_b a$ and assuming that $f(n)$ is sandwiched between two polynomials, we have $T(n) = O(n^{c+1} + f(n))$. (We will not need a function $f(n)$.)

\begin{proposition}
The following is known about inclusions between the classes:
\begin{itemize}
\item NC$^1$ $\subset$ P $\subset$ NP $\subset$ PSPACE $\subset$ NEXPSPACE;
\item NC$^1$ $\subsetneq$ PSPACE;
\item NP $\subsetneq$ NEXPSPACE.
\end{itemize}
\end{proposition}

\subsection{Polynomial space and alternating polynomial time}
\label{sec:AP}


Here, PSPACE is the class of problems that can be solved by a Turing machine in polynomial space (in the input length). 

AP is the set of problems that can be solved by \emph{alternating Turing machines} in polynomial time, meaning the Turing machine can take existential transitions (all states of a nondeterministic Turing machine can be thought of as being existential, and the word is accepted at a particular step if some transition leads to an accepting computation under the stated resource bounds) and universal transitions (in which case the word is accepted if all outgoing transitions lead to an accepting computation). We omit the precise definition, which is analogous to the definition of NP (or Lemma~\ref{lem:PSPACEByTrees} could be taken as a definition).

We will make strong use of the following classical result of Chandra, Kozen and Stockmeyer.

\begin{lemma}[\cite{ChKoSt81}]
\label{lem:AP}
PSPACE = AP
\end{lemma}

Alternating polynomial time has a natural complete problem: Given a word $u$ of length $n$, does a fixed universal Turing machine $T$ accept $u$ in $n$ steps of computation?

Here, the universal machine $T$ is an alternating universal Turing machine that takes a description of any alternating Turing machine, and then simulates its computation, using the same alternation behavior as the simulated machine.

It is helpful in our applications to be a little more explicit about the universal quantification, as we will explicitly lay the tree of universal choices on the configuration, while nondeterministic steps will be nondeterministic (they will be coded in the configuration).

\begin{lemma}
\label{lem:PSPACEByTrees}
There exist two nondeterministic Turing machines $T_0, T_1$ sharing their initial state $q_0$, tape alphabet $S$ and state set $Q$ (but having possibly distinct nondeterministic rules), such that the following problem is PSPACE-hard: Given a word $u = av \in S^n$ with first letter $a \in S$ and with $n$ a power of $2$, does there exist a binary tree of Turing machine configurations $(c_w)_{w \in \{0,1\}^{\leq n}}$ where $c_{\epsilon} = (a \times q_0) v$, $(c_w, c_{w0})$ is a step of $T_0$ and $(c_w, c_{w1})$ is a step of $T_1$, and all $c_w$ with $|w| = n$ are accepting.
\end{lemma}

There are a few notes to make. Note that we are only using $|u|$ time in the computation, and the length of $u$ is a power of $2$. This is nothing to worry about, as in the definition of AP-hardness we ask for a polynomial time computable many-one reduction, and such a reduction is allowed to include a polynomial padding, meaning a polynomial time computation becomes linear. (If the computation does not end at the correct power of $2$ on some branch, the Turing machine can simply wait in the accepting state.)

The point of the two Turing machines is that we can use the steps of $T_0$ and $T_1$ as a universal quantification, as both branches must accept in all leaves. Of course, typically we do not want to universally quantify at every step, but we can simply have $T_0$ and $T_1$ behave the same, i.e.\ according to some alternating Turing machine $T$, except when $T$ makes a universal step, in which case $T_0$ and $T_1$ follow the two choices (or if there are more choices, this can be simulated by a bounded number of branchings).

\subsection{(A version of) Barrington's theorem}
\label{sec:Barrington}

Barrington's theorem \cite{Ba89} relates efficiently parallelizable computation (NC$^1$) and computation with finite memory:

\begin{theorem}[Barrington's theorem]
nonuniform NC$^1$ = BWBP
\end{theorem}

We do not define BWBP, but we now define a ``reversible'' version BWBP', for which a similar theorem holds.

First, we define a permutation group that essentially captures the type of computation performed by BWBP, namely the group of \emph{controlled permutations}. 

\begin{definition}
Let $C, B$ be finite sets and $n \in \N$. Define $G_{n,C,B}$ to be the permutation group acting on $C \times B^n$, generated by the following permutations: for each permutation $\pi$ of $C$ and $b \in B$ and $i \in \{0, \ldots, n-1\}$, the map $\pi_{b,i}$ defined for $|u| = i, b \in B, |v| = n - i - 1, c \in C$ by
\[ \pi_{b,i}(c, ub'v) = \begin{cases}
(\pi(c), ub'v) & \mbox{ if} b' = b, \\
(b, ub'v) & \mbox{ otherwise} \\
\end{cases} \]
We endow $G_{n, A, B}$ with the word norm coming from these generators.
\end{definition}

In English, $\pi_{a,i}$ performs permutation $\pi$ on the $B$-component if the $A^n$-component has symbol $a$ at $i$, and we read it as ``$\pi$ controlled by $a$ at $i$''.

We set up a calculus of controlled permutations. 

For any set $X \subset A^n$ and $\pi \in \Alt(B)$ (even permutations of $B$), define 
\[ \pi|U(b, x) = \begin{cases}
(\pi(b), x) & \mbox{ if} x \in X, \\
(b, x) & \mbox{ otherwise} \\
\end{cases}. \]
We refer to $A^n$-components $u$ of pairs $(b, u) \in B \times A^n$, or sets $X$ used in generators, as the \emph{control}. Note that $\pi_{a, i} = \pi|[a]_i$ where $[a]_i$ is the cylinder $[a]_i = \{uav \;|\; u \in A^i, v \in A^{n-i-1}\}$.





\begin{lemma}
For all $X \subset A^n$ we have $\pi_1|X \circ \pi_2|X = (\pi_1 \circ \pi_2)|X$.
\end{lemma}

\begin{proof}
The $A^n$-component is fixed by elements of the form $\pi|Y$, so the checks $u \in X$ preserve their truth value between applications. Thus if $u \in X$, a direct calculation shows
\[ (\pi_1|X \circ \pi_2|X)(b, u) = (\pi_1(\pi_2(b), u) = (\pi_1 \circ \pi_2)|X(b, u), \]
and if $u \notin X$, then all of $\pi_1|X, \pi_2|X, (\pi_1 \circ \pi_2)|X$ act trivially.
\end{proof}

\begin{lemma}
For all $X \subset A^n$ we have $\pi^{-1}|X = (\pi|X)^{-1}$
\end{lemma}

\begin{proof}
Since $X$ is finite, it suffices to check that $\pi^{-1}|X$ is a one-sided inverse of $\pi|X$. Indeed we have $\pi^{-1}|X \circ \pi|X = (\pi^{-1} \circ \pi)|X = \id|X = \id$ by the previous lemma. 
\end{proof}

We will often use the following lemma about the alternating group (it expresses that the commutator width of the group is one).

\begin{lemma}
\label{lem:Ore}
Let $|B| \geq 5$ be a finite set. Then for all $\pi \in \Alt(B)$ we have $\pi = [\pi_1, \pi_2]$ for some $\pi_1, \pi_2 \in \Alt(B)$.
\end{lemma}

\begin{remark}
The statement of this lemma is nice and clean, but a few remarks on its necessity may be in order, since we are not aware of a proof that avoids case distinctions. The lemma was stated by Ore in \cite{Or51}, and proved by Ito in \cite{It51}. Ito's paper seems difficult to find, but a proof (of a stronger result about even permutations) is found also in \cite{Be72}. It is even known that an analogous result holds (i.e.\ the commutator width is one) in all simple groups. We need the lemma only for a fixed set of cardinalities $|B|$, so one can easily give an ad hoc proof. It even suffices that the commutator width if bounded (over all $|B|$); this changes the degrees of polynomial bounds but none of our main statements. One can also phrase all our arguments directly in terms of $3$-rotations to avoid using the lemma entirely by using simply $[(a \; b \; c), (c \; d \; e)] = (c, b, e)$; indeed this is a common way to prove Barrington's theorem. We also sometimes use only $3$-rotations in the present paper, namely in some cases we need the support to be non-full.
\end{remark}

\begin{lemma}
\label{lem:Intersection}
For all $A, B, n$ and $X, Y \subset A^n$ we have the equality $[\pi_1|X, \pi_2|Y] = [\pi_1, \pi_2]|X \cap Y$.
\end{lemma}

\begin{proof}
If $u \notin X$ then
\begin{align*}
[\pi_1|X, \pi_2|Y](b, u) &= (\pi_1^{-1}|X \circ \pi_2^{-1}|Y \circ \pi_1|X \circ \pi_2|Y)(b, u) \\
&= (\pi_2^{-1}|Y \circ \pi_2|Y)(b, u) \\
&= (b, u) = [\pi, \pi']|X \cap Y(b, u)
\end{align*}
where the second equality follows because the $A^n$-component $u$ is never modified ($\sigma$ is only applied in intermediate steps). A similar calculation works if $u \notin Y$. If $u \in X \cap Y$ then clearly
\[ [\pi_1|X, \pi_2|Y](b, u) = ([\pi_1, \pi_2](b), u) = ([\pi_1, \pi_2]|X \cap Y)(b, u) \]
since all of $u \in X, u \in Y, u \in X \cap Y$ hold at every step.
\end{proof}

The following is an immediate induction.

\begin{lemma}
\label{lem:BigIntersection}
$ [\pi_1|X_1, \pi_2|X_2, \ldots, \pi_k|X_k] = [\pi_1, \pi_2, \ldots, \pi_k]|\bigcap_i X_i $
\end{lemma}

For complexity-theoretic purposes, this is an inefficient definition, and one should balance the brackets instead. However, we only do this for commutators of bounded length, and this simple definition sufficies.

\begin{lemma}
\label{lem:ComplementUnion}
For all $A, B, n$ and $X, Y \subset A^n$ we have 
\[ \pi|\bar{X} = \pi \circ \pi^{-1}|X \]
and
\[ \pi|X \cup Y = \pi|X \circ \pi|Y \circ \pi^{-1}|X \cap Y. \]
\end{lemma}

\begin{proof}
We give semantic proofs. In the first formula, on the right side $\pi$ gets applied anyway, but this cancels when also $\pi^{-1}|X$ applies. Thus, $\pi$ is ultimately applied if and only if the control is not in $X$.

In the second formula, $\pi|X \circ \pi|Y$ applies $\pi$ if the control is in $X$ or $Y$. The only problem is it applies it twice if it is in $X \cap Y$, and this is fixed by also applying $\pi^{-1}|X \cap Y$.
\end{proof}

\begin{lemma}
For all $A, n$ and $X \subset A^n$, if $|B| \geq 5$ then $\pi|X \in G_{n,A,B}$.
\end{lemma}

\begin{proof}
Let $X \subset A^n$ and enumerate $X = \{u^1, u^2, \ldots, u^k\}$ so that $X = \bigcup_{i=1}^k [u^i]_0$. Then $\pi|X = \prod_i \pi|[u^i]_0$. On the other hand for each $u = u_0 u_1 \cdots u_{n-1} \in A^n$, letting $\pi = [\pi_0, \pi_1, \ldots, \pi_{n-1}]$ we have
\begin{align*}
\pi|[u]_0 &= [\pi_0, \pi_1, \ldots, \pi_{n-1}]|\bigcap_{i=0}^{n-1} [u_i]_i \\
&=  [\pi_0|[u_0]_0, \pi_1|[u_1]_1, \ldots, \pi_{n-1}|[u_{n-1}]_{n-1}] \in G_{n,A, B}
\end{align*}
by Lemma~\ref{lem:BigIntersection}.
\end{proof}

The lengths of some controlled permutations are necessarily very long, since the group is very large. BWBP' is defined through controlled permutations with small word norm.

\begin{definition}
A language $L \subset A^*$ is in BWBP'$_B$ if for some $d$, for all $n$ and $\pi \in \Alt(B)$, the norm of $\pi|L \cap A^n$ in $G_{n,A,B}$ is at most $n^d + d$. Write
\[ \mathrm{BWBP'} = \bigcup_{B\; \mathrm{ finite}} \mathrm{BWBP'}_B. \]
\end{definition}

(Writing $n^d + d$ here is equivalent to writing $O(n^d)$.)

We now prove the relevant part of Barrington's theorem.

\begin{theorem}[variant of Barrington's theorem]
\label{thm:BarringtonVariant}
nonuniform NC$^1$ $\subset$ BWBP'$_B$ for all $|B| \geq 5$
\end{theorem}

\begin{proof}
Let $A = \{0, 1\}$. Let $B$ be any alphabet with cardinality at least $5$. Let $s = (s_1, s_2) : \Alt(B) \to \Alt(B) \times \Alt(B)$ be any function such that $\pi = [s_1(\pi), s_2(\pi)$ for all $\pi$, which exists by Lemma~\ref{lem:Ore}. Identify Boolean formulas $\phi$ with the subsets of $A^n$ they describe. Then by Lemma~\ref{lem:Intersection} and Lemma~\ref{lem:ComplementUnion} we have
\[ \pi|\phi_1 \wedge \phi_2 = [s_1(\pi)|\phi_1, s_2(\pi)|\phi_2] \]
\[ \pi|\phi_1 \vee \phi_2 = \pi|\phi_1 \circ \pi|\phi_2 \circ [s_1(\pi)^{-1}|\phi_1, s_2(\pi)^{-1}|\phi_2] \]
\[ \pi|\neq \phi = \pi \circ \pi^{-1}|\phi \]

Let $T_d$ be the maximal word norm of $\pi|\phi$ for $\phi$ of depth $k = O(\log n)$. The formulas show $T_i \leq 6T_{i-1}$ for all $i$, so $T_k = O(6^k) = O(6^{\log n}) = O(n^d)$ for some $d$.
\end{proof}




We now give the main examples of languages are in BWBP', by showing they are in NC$^1$. 

For the next lemma, let $S$ be the tape alphabet of a Turing machine $T$, and $Q$ its state set. Let $C_{S, Q} \subset ((S \times Q) \cup S)^*$ be the language of words where a single $S \times Q$-symbol appears. In other words, the language of Turing machine configurations.

\begin{lemma}
\label{lem:TuringStep}
The language $L = \{u \# v \;|\; (u, v) \mbox{ is a computation step }\} \subset ((S \times Q) \cup S \cup \{\#\})^*$ is in BWBP', for any Turing machine $T$ with tape alphabet $S$ and state set $Q$.
\end{lemma}

\begin{proof}
Let $A = ((S \times Q) \cup S \cup \{\#\})^*$. We show that a logarithmic-depth formula exists for inputs $w = uav$ of length $2n+1$ with $u \in A^n, a \in A, v \in A^n$, i.e.\ that the language is in NC$^1$. Note that Boolean formulas can only talk about bits, but we can take the basic Boolean variables to be statements of the form ``$w_i = a$'' for $i \in [0, 2n]$. For $I \subset \{0,\ldots, 2n+1\}$, we use names of the form $\texttt{FORMULA}_I(x_1, \ldots, x_k)$ for Boolean formulas. Formally, this is just a name, but the label \texttt{FORMULA} and the parameters $(x_1, \ldots, x_k)$ explain what the formula is describing, and the subscript $I$ means that the formula ultimately only refers to coordinates $I$ of the input $w$. Omitting the subscript means $I = [0, 2n]$.

The point of $I$ is that logarithmic depth is obvious as soon as we always roughly cut the size of $I$ in half. We will describe the bottom of the recursion, and only give the formulas that apply for large $I$, since for small $I$ (say $|I| \leq 4$, there exists a formula of constant size by the simple fact that Boolean formulas can describe any set.  

First, we define formulas checking $a = \#$ (i.e.\ the $n$th coordinate of $w$ contains $\#$) and $\# \not\sqsubset uv$. For this, define $\texttt{SUBALPHABET}_I(B)$ to be the check that $w_i \in B$ for all $i \in I$. The formula $\texttt{SUBALPHABET}_I(B) = \texttt{SUBALPHABET}_J(B) \wedge \texttt{SUBALPHABET}_K(B)$ shows that this formula always has logarithmic depth, by using cuts $I = J \sqcup K$ with $|J| \approx |K| \approx |I|/2$. Now the desired formula is
\[ \texttt{HASHCHECK}() = \texttt{ALPHABET}_{\{n\}}(\{\#\}) \wedge \texttt{ALPHABET}_{[0,n-1] \cup [n+1,2n]}(A \setminus \{\#\}). \]

Next, we check the main part, i.e.\ that a single Turing machine step is being described. We write $\texttt{STEP}()$ for this formula, where $I$. It suffices to show for each $(q, s, q', s', d) \in \Delta$ we have a logarithmic depth formula for $\texttt{STEP}(q, s, q', s', d)$, which accepts precisely the steps of type $(q, s, q', s', d)$. Namely then
\[ \texttt{STEP}() = \bigvee_{(q, s, q', s', d) \in \Delta} \texttt{STEP}(q, s, q', s', d). \]

We only show how to define $\texttt{STEP}(q, s, q', s', 1)$, as $\texttt{STEP}(q, s, q', s', 0)$ is strictly easier and $\texttt{STEP}(q, s, q', s', -1)$ is completely symmetric. 

We need some auxiliary firmulas. As a first auxiliary formula, we define
$\texttt{BASICSTEPAT}_{\{i, i+1, i+n+1, i+n+2\}}(i, q, s, q', s', 1)$ as the check that 
\[ w_iw_{i+1}w_{i+n+1}w_{i+n+2} = ((q, s), a, s', (q',a')) \]
for some $a, a' \in S$. Since this formula talks about only $4$ positions of the input, it has a formula of constant size.

As another auxiliary formula, for $J, K \subset [0,2n+1]$ with $|J| = |K|$, define $\texttt{EQ}_{J \sqcup K}(J, K)$ as the check that $w_j = w_k$ when $j$ and $k$ are the $i$th smallest element of $J,K$ respectively (i.e.\ the subsequences along $J$ and $K$ are equal). We have
\[ \texttt{EQ}_{J \cup K}(J, K) = \texttt{EQ}_{J' \cup K'}(J', K') \wedge \texttt{EQ}_{J'' \cup K''}(J'', K''), \]
where $J = J' \sqcup J''$ and $K = K' \sqcup K''$ are roughly equal cuts at the middle with $|J'| = |K'|$. Then clearly $\texttt{EQ}_{J \cup K}(J, K)$ always has a logarithmic depth formula.

We now observe that
\begin{align*} \texttt{STEP}(q, s, q', s', 1) = \bigvee_{i \in [0, n-2]} (&\texttt{BASICSTEP}_{{i, i+1, i+n+1, i+n+2}}(i, q, s, q', s', 1) \\
&\wedge \texttt{EQ}_{J_i \cup K_i}(J_i, K_i))
\end{align*}
where $J_i = [0, n-1] \setminus \{i,i+1\}, K_i = [n+1, 2n] \setminus \{i+n+1,i+n+2\}$.
Finally, a logarithmic depth formula for a polynomial-sized disjunction $\bigvee_{i \in [0, n-2]} \phi_i$ is obtained by associating it in a balanced way.

We conclude that there is a formula
\[ \texttt{L}() = \texttt{STEP}() \wedge \texttt{HASHCHECK}() \]
of logarithmic depth in $n$, which describes $L$.

By Theorem~\ref{thm:BarringtonVariant}, $\pi|L$ has polynomial norm in $G_{n, A, B}$ for $|B| \geq 5$.
\end{proof}

Having written one proof in detail, we will from now on not explicitly write the formulas, and only explain why a divide-and-conquer approach is possible. In each case, we can split the problem into a bounded number of subproblems, where the number of bits being looked at becomes roughly cut in half.

\begin{lemma}
\label{lem:FixedCheck}
For any sequence of words $u_i \in \Sigma^i$, the language $\{u_i \;|\; i \in \N\}$ is in BWBP'.
\end{lemma}

\begin{proof}
To compare $w \in \Sigma^i$ with $u_i$, split $w$ in half, and recursively compare the left and right halves.
\end{proof}

For the next lemma for $n \in [0, m]$, denote by $\unary_m(n) = 0^n 1^{m-n}$ the length-$m$ unary representation of $n$.

\begin{lemma}
\label{lem:Increment}
The language $\{\unary_m(n) \# \unary_m(n+1) \;|\; n \in [0, m) \} \subset \{0,1,\#\}^*$ is in BWBP'. Similarly, the relation $\{u2^n \# ua2^{n-1} \;|\; u \in \{0,1\}^*, n \in \N\}$ is in BWBP' for each $a \in \{0,1\}$.
\end{lemma}

\begin{proof}
First, consider th first language. For simplicity assume $m$ is a power of $2$. For the divide-and-conquer, observe that checking the input to be in the form $\unary_m(n) \# \unary_m(n+1)$ can be split in two cases after splitting the input in half:
\begin{itemize}
\item the left halves are all zero, and the right halves are in the increment relation;
\item the right halves are all one, and the left halves are in the increment relation.
\end{itemize}

Consider now the second language. It is in BWBP', as it suffices to first check that the projection $0,1 \mapsto 0; 2 \mapsto 1$ give a word in the first language, then check that binary symbols appear in the same positions in the words on each side of $\#$, or in case only the second word carries a bit, this bit is $a$. (All but the first check are a matter of checking that all pairs of corresponding positions in the two parts belong to a particular set.)
\end{proof}

The increment relation for binary numbers is also in BWBP'. Say numbers are written least significant bit first and $0$-indexed, i.e.\ define $\val(u) = \sum_{i=0}^{|u|-1} 2^i u_i$ for a word $u \in \{0,1\}^*$. Let $\bin_m : [0, 2^m) \to \{0,1\}^m$ be the inverse of this function, i.e.\ it gives the binary representation of a natural number.

\begin{lemma}
\label{lem:Increment}
The language $\{\bin_m(n) \bin_m(n+1) \;|\; n \in [0, 2^m) \} \subset \{0,1\}^*$ is in BWBP'.
\end{lemma}

\begin{proof}
Again for simplicity assume $m$ is a power of $2$. For the divide-and-conquer, observe that $\bin_m(n) \# \bin_m(n+1)$ can be split in finitely many cases:
\begin{itemize}
\item the left halves are equal, and the right halves are in the increment relation
\item the right halves are respectively $1^{m/2}$ and $0^{m/2}$, and the left halves are in the increment relation. \qedhere
\end{itemize}
\end{proof}

More generally, checking for a fixed constant sum can be performed in polynomial word norm.

\begin{lemma}
\label{lem:Sum}
Fix $d$, $|C| \geq 5$ and any $\pi \in \Alt(C)$. Then for any $\ell$, the permutation $\pi|\{\bin_m(n_1) \bin_m(n_2) \cdots \bin_m(n_d) \;|\; \sum_i n_i = \ell \}$ has word norm polynomial in $m$ in $\mathcal{G}_{dm+d-1,C,B}$.
\end{lemma}

\begin{proof}
It suffices to construct a circuit whose output is the sum of two numbers, as one can chain $d-1$ of these together and finally compare the sum to a fixed number using Lemma~\ref{lem:FixedCheck}.

This is \cite[Example 6.19]{ArBa09}. The bits of the sum are determined by the added bits at that position, as well as the carry signal coming from the left (from the left, since binary numbers are written least significant bit first). It is easy to show that the carry signal at a particular position can in turn be ``predicted'' by a constant-depth circuit, when we allow conjunctions and disjunctions (circuit nodes labeled $\wedge$ and $\vee$) with arbitrarily many inputs. Associating these in a balanced way gives a logarithmic depth circuit where all nodes have in-degree $2$.
\end{proof}


\end{document}